\documentclass[a4paper,reqno]{amsart}
\title[Hopf--Galois structures and skew bracoids of small degree]{Constructing Hopf--Galois structures and skew bracoids of small degree}
\author{Andrew Darlington}
\date{\today}
\address{Department of Mathematics and Statistics, Faculty of Environment, Science and Economy, University of Exeter, Exeter EX4 QFU. UK.}
\address{Department of Mathematics, Vrije Universiteit Brussel, Pleinlaan 2,
1050 Brussel, Belgium}
\email{andrew.darlington@vub.be}

\author{E.A.\  O'Brien}
\address{Department of Mathematics, University of Auckland, Private
  Bag 92019, Auckland, New Zealand} 
\email{e.obrien@auckland.ac.nz}

\thanks{For the purpose of open access, the authors have applied a CC BY 
public copyright license to any Author Accepted Manuscript version arising.
\newline
\indent Data Access Statement: Data sharing is not applicable to this article 
as no datasets were generated or analysed in this research.
}

\newif\ifversion
\versiontrue   

\usepackage[utf8]{inputenc}
\usepackage{amsmath}
\usepackage{amsfonts}
\usepackage{amssymb}
\usepackage{amsthm}
\usepackage{pifont} 
\usepackage{enumerate} 
\usepackage{enumitem}
\usepackage{wasysym} 
\usepackage{mathrsfs}
\usepackage{xfrac} 
\usepackage{array}
\usepackage{siunitx}
\usepackage{mathtools}
\usepackage{framed}
\usepackage{pifont}
\usepackage[utf8]{inputenc}
\usepackage[english]{babel}
\usepackage[utf8]{inputenc}
\usepackage{tikz,tkz-euclide}
\usepackage{xcolor,graphicx}
\usepackage{faktor}
\usepackage{xfrac}
\usepackage{float}
\usepackage{tikz-cd}
\usepackage[figuresleft]{rotating}
\usepackage{multirow}
\usepackage{listings}
\usepackage{url}
\usepackage{pdflscape}

\makeatletter
\def\bign#1{\mathclose{\hbox{$\left#1\vbox to8.5\p@{}\right.\n@space$}}\mathopen{}}
\def\Bign#1{\mathclose{\hbox{$\left#1\vbox to11.5\p@{}\right.\n@space$}}\mathopen{}}

\newtheorem{theorem}{Theorem}[section]
\newtheorem{proposition}[theorem]{Proposition}
\newtheorem{lemma}[theorem]{Lemma}
\newtheorem{corollary}[theorem]{Corollary}
\newtheorem{conjecture}[theorem]{Conjecture}

\theoremstyle{definition}
\newtheorem{remark}[theorem]{Remark}
\newtheorem{example}[theorem]{Example}
\newtheorem{definition}[theorem]{Definition}

\newcommand{\Gal}{\mathrm{Gal}}
\newcommand{\Hol}{\mathrm{Hol}}
\newcommand{\Aut}{\mathrm{Aut}}
\newcommand{\Perm}{\mathrm{Perm}}
\newcommand{\Stab}{\mathrm{Stab}}

\newcommand{\C}{\mathrm{C}}
\newcommand{\D}{\mathrm{D}}

\begin{document}
\bibliographystyle{amsalpha}
\begin{abstract}
Using the fact that Hopf-Galois structures on separable
extensions and skew bracoids are both intrinsically connected to
transitive subgroups of the holomorph of a finite group, 
we present algorithms to classify and enumerate these objects for 
small degree, and apply them to obtain 
significant extensions to existing results.
We also explore the classifications of these structures of degree $2pq$, 
where $p$ and $q$ are distinct odd primes. We conclude with some 
enumeration-inspired observations and a conjecture.
\end{abstract}

\maketitle

\section{Introduction}
Hopf--Galois theory and skew bracoids are two, initially seemingly
disconnected, areas of algebra which have important applications
in mathematics and physics.

Hopf--Galois structures were introduced by Chase and Sweedler
\cite{CS69}; they investigated and generalised what it means
for a field extension $L/K$ to have Galois group $J$. We 
associate to $L/K$, which need not be Galois in the classical sense,
a $K$-Hopf algebra $H$ which acts on $L$ via $\cdot$ in such a way
that it behaves like a Galois group. The pair $(H,\cdot)$ is
a \textit{Hopf--Galois structure} on $L/K$. While there is
only one Galois group associated to a Galois extension, $L/K$ may
admit several 
Hopf--Galois structures. 
Since the single extension can now be viewed through several lenses, 
this has important applications, including to Galois module theory.
One important notion is that of the
Hopf--Galois correspondence. As in the classical Galois
correspondence, given a Hopf--Galois structure $(H,\cdot)$ on $L/K$, the Hopf
subalgebras of $H$ correspond to intermediate fields of $L/K$. This
correspondence is always injective but not necessarily surjective.

A \textit{brace} is a triple $(B,+,\circ)$, where $B$ is a set 
and $+$ and $\circ$ are 
operations such that $(B,+)$ is an abelian group and $(B,\circ)$ is a group, 
and the elements of $B$ satisfy the following `skew-distributivity' relation:
\[a\circ(b+c)=a\circ b-a+a\circ c \;\;\; \forall a,b,c \in B.\]
The {\it order} of the brace is $|B|$. 
(While $B$ need not be finite, we assume so here.) 
Some of the fundamental properties of braces were formulated by 
Rump \cite{Rum07}. 
By relaxing the condition that $(B,+)$
is abelian, 
Guarnieri and Vendramin \cite{GV17}
introduced {\it skew braces}. These provide non-degenerate solutions 
to the set-theoretic Yang-Baxter equation (ST-YBE). 
One motive behind understanding solutions to the ST-YBE is to shed more
light on the full version of the Yang-Baxter equation. 
It arises in many contexts, 
including statistical mechanics and representations of the braid group.

Skew bracoids 
were introduced 
by Martin-Lyons and Truman \cite{MLT24}. 
They offer a broader perspective of skew braces; instead of a set with 
two binary operations, a skew bracoid comprises two groups 
$(N,+)$ and $(G,\circ)$ with a transitive action of $G$ on $N$ 
which mimics skew brace behaviour.  The groups need not have
the same order. 
Hence, as we discuss later, the notion of `order' no longer makes 
sense for skew bracoids; instead, we refer to $|N|$ as 
the {\it degree} of the skew bracoid.

Colazzo, Koch,
Martin-Lyons and Truman \cite{CKMLT24} showed that certain classes of skew
bracoids yield right (but not necessarily left)
non-degenerate solutions to the ST-YBE.

The intimate connection between Hopf--Galois
structures on Galois/separable extensions and skew braces/bracoids
has been explored in a series of papers. On the one hand, given a separable extension $L/K$
with Galois closure $E$, if $J=\Gal(E/K)$ and
$J'=\Gal(E/L)$, then Greither and Pareigis \cite{GP87} showed that
the Hopf--Galois structures on $L/K$ 
correspond to certain regular
subgroups $N$ of $\Perm(J/J')$, the symmetric group on the set of 
left cosets of $J'$ in $J$. Therefore $|N|$ is 
the degree of $L/K$, and $N$ is the \textit{type} of the
associated Hopf--Galois structure. 
Byott \cite{Byo96} showed that 
separable extensions admitting
Hopf--Galois structures of type $N$ 
correspond to transitive subgroups of 
$\Hol(N) \cong N \rtimes \Aut(N)$, the holomorph of $N$. 
On the other hand, it was shown in \cite{GV17} that skew
braces with `additive group' isomorphic to $N$ correspond to regular
subgroups of $\Hol(N)$. The connection between skew braces and
Hopf--Galois structures on Galois extensions via regular subgroups of
the holomorph was explored 
in \cite{SV18}. 
Martin-Lyons and Truman \cite{MLT24}
showed that skew bracoids
correspond to transitive subgroups of the holomorph, and are therefore
closely connected to Hopf--Galois structures on separable
extensions.

The connections between both structures
makes the task of enumerating and classifying them 
particularly amenable to algorithms and computation.
One motivation for doing so is to produce a database 
of examples, so permitting their in-depth study.
Guarnieri and Vendramin \cite[Algorithm 5.1]{GV17} 
computed the number of skew braces of order at most $30$. 
By exploiting the observation that two regular
subgroups of $\Hol(N)$ are conjugate if and only if they are conjugate
by an element of $\Aut(N)$,  
Bardakov, Neshchadim and Yadev \cite{BNY20}
enumerated, with some exceptions, the skew braces of order at most 868.
The first computer-aided enumeration of Hopf--Galois structures on Galois
extensions was done by Byott and Vendramin in the appendix of
\cite{SV18}. Crespo and Salguero \cite{CS20} gave
an algorithm which they used to enumerate the Hopf--Galois structures on 
separable extensions of degree at most $11$ and in \cite{CS21} extended these 
results to degree $31$.

In this paper, we present new algorithms to enumerate and classify
both Hopf--Galois structures on separable extensions and skew
bracoids. Given a finite group $N$, 
we use the computational algebra system {\sc Magma} \cite{BCP97} to 
compute the transitive subgroups of $\Hol(N)$, and then sort these
into relevant classes to classify the corresponding
structures.  Unlike \cite{CS20,CS21}, our approach does 
not rely on the classification of all transitive permutation groups
of a given degree; these are known up to degree 48 \cite{Deg48}.
Using our implementations of these algorithms
and extensive computing resources, we obtained significant
extensions to existing results. 
We enumerated Hopf--Galois structures 
on separable extensions and skew bracoids up to degree 200, 
excluding those of degree 64, 96, 128, 144, 160, 162 and 192.  
The resources required for classification are significantly greater. 
Consequently, we classified the structures only up to degree 100, excluding 
those of degree 32, 48, 64, 80, 81 and 96.  The principal
limitation to the enumeration algorithm is constructing the 
transitive subgroups of the relevant holomorphs; our classification 
algorithm must also solve challenging isomorphism problems.
Detailed results are available at \cite{GITPAGE} and are 
recorded in a format which permits both their ready access and further 
study within {\sc Magma}.  Our implementations can also be used 
for additional or more focused explorations.

The structure of the paper is the following.
We review necessary preliminaries in Section \ref{prelim}. 
Our enumeration and classification algorithms are described
in Section \ref{alg}, where we highlight the new 
ideas and explain how both the relevant Hopf--Galois 
structures and skew bracoids may be recovered from the data. 
In Section \ref{results}, we present our enumeration and structure results.
Motivated by some of these results, 
in Section \ref{ord2pq} we study 
Hopf--Galois structures and skew bracoids of degree $2pq$ where $p$ and $q$ 
are distinct odd primes. Finally, we present some observations
and a conjecture inspired by our data. 

\section{Preliminaries}\label{prelim}
In this section, we review the necessary definitions and results
relating to Hopf--Galois structures, skew bracoids and transitive
subgroups. 

\subsection{Transitive subgroups of the holomorph}\label{transgps}
Unless otherwise stated, let $n$ be a positive integer and 
let $N$ be a group of order $n$.
\begin{definition}
    A subgroup $M$ of $\Perm(N)$ is \textit{transitive} if $M$ acts
    transitively on $N$.  A transitive subgroup $M$ is 
    \textit{regular} if $|M|=|N|$, or, equivalently, if
    $\Stab_M(1_N)=1_M$.
\end{definition}
Observe that 
we may view $N$ inside $\Perm(N)$ as the image of the left translation
map \mbox{$\lambda:N \to \Perm(N)$}. 

\begin{definition}
    The \textit{holomorph} $\Hol(N)$ of $N$ is the normaliser of
    $\lambda(N)$ in $\Perm(N)$. 
    \[\Hol(N)=\mathrm{Norm}_{\Perm(N)}(\lambda(N))=\{\pi \in \Perm(N) \mid \pi\lambda(\eta)\pi^{-1} \in \lambda(N) \; \forall \eta \in N\}.\]
\end{definition}
For notational convenience, 
we sometimes write $\lambda(\eta)$ as $\lambda_\eta$.
In hand computation, we often use the 
observation that $\Hol(N) \cong N \rtimes \Aut(N)$.
Multiplication in $\Hol(N)$ is defined as follows: for
$\eta,\mu \in N$ and $\alpha,\beta \in \Aut(N)$, 
    \[(\eta,\alpha)(\mu,\beta)=(\eta\alpha(\mu),\alpha\beta).\]
Every subgroup $M$ of $\Hol(N)$ 
induces an action $\cdot$ on $N$ given by
\[(\eta,\alpha)\cdot \mu = \eta\alpha(\mu).\]
The following is a slight adaption of a key observation in \cite{BNY20}.
\begin{proposition}\label{trans_conj}
    Let $M_1$ and $M_2$ be subgroups of $\Hol(N)$ such that $M_1$ is
    transitive on $N$. Then $M_1$ is conjugate to $M_2$ in $\Hol(N)$
    if and only if they are conjugate by an element of $\Aut(N)$.
\end{proposition}
\begin{proof}
    Let $M_1$ be a transitive subgroup of $\Hol(N)$, and let 
    $(\eta,\alpha) \in \Hol(N)$ such that $M_2=(\eta,\alpha)^{-1}M_1(\eta,\alpha)$. We show that there exists
    $\beta \in \Aut(N)$ such
    that \[(1,\beta)^{-1}M_1(1,\beta)=(\eta,\alpha)^{-1}M_1(\eta,\alpha)=M_2.\]
    To this end, since $M_1$ acts transitively on $N$, there 
    exists $(\mu,\gamma) \in M_1$ such that $(\mu,\gamma)\cdot \eta =
    \mu\gamma(\eta)=1$. Therefore
    \[(\mu,\gamma)(\eta,\alpha)=(1,\gamma\alpha).\]
    If $\beta=\gamma\alpha \in \Aut(N)$, then 
    \begin{align*}
        (1,\beta)^{-1}M_1(1,\beta)&=((\mu,\gamma)(\eta,\alpha))^{-1}M_1(\mu,\gamma)(\eta,\alpha)\\ 
&=(\eta,\alpha)^{-1}(\mu,\gamma)^{-1}M_1(\mu,\gamma)(\eta,\alpha) \\
&=(\eta,\alpha)^{-1}M_1(\eta,\alpha)\\
&=M_2. \qedhere
    \end{align*} 
\end{proof}
\begin{corollary}\label{conj_isom}
    Let $M_1$ and $M_2$ be subgroups of $\Hol(N)$ such 
that $M_1$ is transitive on
    $N$. If $M_1$ and $M_2$ are conjugate in $\Hol(N)$, then they
    are permutation isomorphic and so, in particular, $M_2$
    is also transitive on $N$.
\end{corollary}
\begin{proof}
    If $M_1$ and $M_2$ are conjugate in $\Hol(N)$, then
    by Proposition \ref{trans_conj} there exists 
    $\beta \in \Aut(N)$ such that $\beta M_1 \beta^{-1} =
    M_2$. Further, since $\beta(1_N)=1_N$, clearly 
    $\beta\Stab_{M_1}(1_N)\beta^{-1} \leq
    \Stab_{M_2}(1_N)$. If $g_2 \in \Stab_{M_2}(1_N)$,
    then there exists $g_1 \in M_1$ such that $\beta g_1
    \beta^{-1}=g_2$,
    Hence $g_1 \in
    \Stab_{M_1}(1_N)$ and
    $\beta\Stab_{M_1}(1_N)\beta^{-1} =
    \Stab_{M_2}(1_N)$. Thus conjugation by $\beta$ is an
    isomorphism of permutation groups.
\end{proof}

\subsection{Hopf--Galois structures}\label{HGS}
Let $L/K$ be a field extension and let $H$ be a $K$-Hopf algebra
acting on $L$ with action $\cdot$ such that $L$ is an $H$-module
algebra
(see, for example, \cite[Chapter 1, \S2]{Chi00}). 

We say that $H$ with its action on $L$ gives a
\textit{Hopf--Galois structure} on $L/K$ if $\cdot$ induces an
isomorphism between the $K$-vector spaces $L \otimes_K H$ and
$\mathrm{End}_K(L)$. If, in addition, $L/K$ is separable, then denote
its Galois closure by $E$, and define the Galois groups $J=\Gal(E/K)$
and $J'=\Gal(E/L)$. Consider the left translation map $\lambda:J
\rightarrow \Perm(J/J')$, $\lambda(g)(\overline{h})=\overline{gh}$. It
can be shown that $\lambda$ is injective, and so we may identify $J$
with a permutation group of degree $n=[L:K]$ 
(where $J'$ is identified with $\lambda(J')$).
Greither and Pareigis 
\cite{GP87} showed that if $N\leq\Perm(J/J')$ acts regularly on $J/J'$ and is
normalised by $\lambda(J)$, then $L/K$ admits a Hopf--Galois structure
$H=E[N]^J$ of type (the isomorphism class of) $N$, with $L-$action
given by 
\begin{equation}\label{H_action}
\left(\sum_{\eta \in N}a_{\eta}\eta\right)\cdot x = \sum_{\eta \in
  N}a_{\eta}(\eta^{-1}(1_{J/J'}))(x).
\end{equation}
This correspondence is bijective, so every Hopf--Galois
structure on $L/K$ arises in this way. 
However $\Perm(J/J')$ grows quickly, so it may be computationally
infeasible to find all the relevant subgroups. 
The following theorem of Byott \cite{Byo96} 
greatly reduces our search space; we use the formulation of 
\cite[Chapter 2]{Chi00}.
\begin{theorem}[Byott's translation]\label{Byo_trans}
    There is a bijection between the following sets:
    \[\mathcal{N}=\{\alpha:N \rightarrow \Perm(J/J') \mid \alpha 
\text{ injective homomorphism with } \alpha(N) \text{ regular}\}\]
    and
    \[\mathcal{J}=\{\beta:J\rightarrow \Perm(N) \mid \beta \text{ 
injective homomorphism with } \beta(J')=\Stab_{\beta(J)}(1_N)\}.\]
    If $\alpha,\alpha' \in \mathcal{N}$ correspond to
    $\beta,\beta' \in \mathcal{J}$, then $\alpha(N)=\alpha'(N)$ if and
only if
    $\beta(J)$ and $\beta'(J)$ are conjugate by some $\phi \in
    \Aut(N)$; and $\alpha(N)$ is normalised by the image of the 
left translation map $\lambda: J \to \Perm(J/J')$ if and 
only if $\beta(J)$ is contained in $\Hol(N)$.
\end{theorem}
A Hopf--Galois structure of type $N$ \textit{realises} the pair $(J, J')$ 
if there is a transitive subgroup $M$
of $\Hol(N)$ and an isomorphism $\phi:J \to M$ such 
that $\phi(J')=\Stab_M(1_N)$.
Theorem \ref{Byo_trans} underpins 
the following counting result.
\begin{lemma}[\cite{Byo96}]\label{Byott_num_HGS}
    Let $J,J'$ and $N$ be as above. Let $e(J,J',N)$ be the number of
    Hopf--Galois structures of type $N$ which realise $(J,J')$, and
    let $e'(J,J',N)$ be the number of transitive subgroups $M$ of
    $\mathrm{Hol}(N)$ isomorphic to $J$ via an isomorphism 
taking $\Stab_M(1_N)$ to $J'$. Now 
    \[e(J,J',N)=\frac{|\mathrm{Aut}(J,J')|}{|\mathrm{Aut}(N)|}e'(J,J',N),\]
    where
    \[\Aut(J,J')=\left\{\theta\in\Aut(J) \mid \theta(J')=J'\right\},\]
    the group of automorphisms $\theta$ of $J$ such that $\theta$
    fixes the identity coset $1_JJ'$ of $J/J'$.
\end{lemma}
Let $(H,\cdot)$ give a Hopf--Galois structure on a field extension
$L/K$, let $\varepsilon$ be the counit of $H$ and let $I$ be a Hopf subalgebra of $H$. The subfield
of $L$ fixed by $I$ is
\[L^I=\{x \in L \mid h\cdot x=\varepsilon(h)x \; \forall h \in I\}.\]
In particular, $L^H=K$. In \cite{CS69} it was shown that
this `Hopf--Galois correspondence' is injective and inclusion
reversing, but it is not necessarily surjective, unlike the usual
Galois correspondence. However, certain separable
extensions admit at least one
Hopf--Galois structure for which the Hopf--Galois correspondence is
surjective. A separable extension $L/K$ is \textit{almost
classically Galois} if $J'$ has a normal complement in $J$;
equivalently, there is a regular subgroup $N$ of $\Perm(J/J')$
normalised by $J$ and contained in $J$. Therefore, if $L/K$ is almost
classically Galois, then $J \cong N \rtimes J'$, and 
the Hopf algebra $H=E[N^{opp}]^J$ gives a Hopf--Galois structure
on $L/K$ admitting a bijective correspondence. Here, $N^{opp}$ denotes
the opposite group of $N$: that is, $N^{opp}$ has underlying set $N$,
and if $*$ and $*_{opp}$ are the operations on $N$ and $N^{opp}$
respectively, then $g *_{opp} h = h * g$ for all $g,h \in N$. Although
the property of being almost classically Galois was initially applied to
the field extension, Kohl \cite{Kohl98} refers to the Hopf--Galois
structure given by $E[N^{opp}]^J$ as an \textit{almost classically
Galois structure}. 
A natural question is: how many
such structures does an almost classically Galois
extension admit? That 
$L/K$ is almost classically Galois does not guarantee that all
Hopf--Galois structures on $L/K$ are almost classically Galois, nor
that all such admit a bijective correspondence.

\subsection{Skew bracoids}\label{sboids}
A \textit{skew bracoid} is a quintuple $(N,+,G,\circ,\odot)$ such that
$(N,+)$ and $(G,\circ)$ are (not necessarily abelian) groups with a
transitive action satisfying the following relation:
\[g \odot (\mu+\eta)=(g\odot \mu)-(g\odot 1_N)+(g\odot \eta) 
\;\; \forall g \in G \mathrm{\ and\ } \forall \eta, \mu \in N.\]
If the group operations on $G$ and $N$ are clear, then we 
denote $(N,+,G,\circ,\odot)$ by $(G,N,\odot)$.
\begin{remark} 
    If $|G| = |N|$, then a skew bracoid is \textit{essentially} a
    skew brace. As noted in \cite{MLT24}, the structure of
    $G$ may be transported to $N$ via the rule
    \[(g \odot 1_N)\circ(h \odot 1_N)=(gh) \odot 1_N.\]
    Now $(N,+,\circ)$ is a skew brace. Define the map $r:N \times N \rightarrow N \times N$ by
    \[r(x,y)=(-x+x \circ y,(-x+x \circ y)^{-1}\circ x \circ y)\]
    for all $x,y \in N$. 
Both components of $r$ are bijective functions, and $r$ satisfies
    \[(r \times \mathrm{id})(\mathrm{id}\times r)(r \times \mathrm{id})=(\mathrm{id}\times r)(r \times \mathrm{id})(\mathrm{id}\times r).\]
Hence $r$ gives a non-degenerate solution to the ST-YBE on the set $N$.
As mentioned above, certain families of
    skew bracoids yield other types of
    solutions to the ST-YBE; we omit the details. 
\end{remark}
Recall that the size of the underlying set of a skew brace is typically 
called the \textit{order} of the brace. 
This no longer makes sense for skew bracoids since we now consider 
two sets of possibly different sizes. 
Instead, we define the \textit{degree} of the skew bracoid $(G,N,\odot)$ to be 
the order of $N$: this better reflects the relationship of $(G,N,\odot)$ to the 
corresponding permutation group of degree $|N|$. If the skew 
bracoid is essentially a skew brace, then the notions 
of `degree' and `order' coincide.

To classify these structures, we must understand them in
more detail. Suppose we have skew bracoids
$(G_1,N,\odot_1)$ and $(G_2,N,\odot_2)$, and let
$\lambda_{\odot_1}:G_1 \to \Perm(N)$ and $\lambda_{\odot_2}:G_2 \to
\Perm(N)$ be the permutation representations determined by the 
actions $\odot_1$ and $\odot_2$ respectively. 
It was shown in \cite[Theorem 2.8]{MLT24}
that $\lambda_{\odot_i}(G_i)$ is contained in $\Hol(N)$ and 
is isomorphic to a quotient of $G_i$. 
If $\lambda_{\odot_1}(G_1)=\lambda_{\odot_2}(G_2)$, then 
$G_1$ and $G_2$ both act on $N$ in `essentially the same' way. If
$(G_1,N,\odot_1)$ and $(G_2,N,\odot_2)$ are related in this
way, then they are \textit{equivalent}. This reduces
to equality in the case of skew braces.

The following is \cite[Corollary 2.21]{MLT24}.
\begin{proposition}
    Given a group $N$, there is a bijection between transitive
    subgroups of $\Hol(N)$ and equivalence classes of skew bracoids
    $(G,N,\odot)$.
\end{proposition}
Every skew bracoid $(G,N,\odot)$ is equivalent to a skew
bracoid $(G',N',\odot')$ where $G'$ acts faithfully on $N'$ 
(so the permutation representation of $\odot'$ has trivial
kernel). A skew bracoid with such a faithful transitive action is
\textit{reduced}. 

\begin{definition}
    An \textit{isomorphism} of skew bracoids
    $(G_1,N_1,\odot_1)\to(G_2,N_2,\odot_2)$ is a pair of group isomorphisms
    \begin{align*}
        &\phi:G_1\to G_2,\\ &\psi:N_1\to N_2
    \end{align*}
    such that
    \[\psi(g \odot_1 \eta)=\phi(g)\odot_2 \psi(\eta) \;\; \forall g \in G_1, \eta \in N_1.\]
\end{definition}
If $(\phi,\psi)$ is an isomorphism of skew
bracoids, then $\psi$ is completely determined by
$\phi$.
We need one more result \cite[Proposition 4.13]{MLT24},
which emulates skew brace theory.
\begin{proposition}
    Let $N$ be a group and let $(G_1,N,\odot_1)$ and $(G_2,N,\odot_2)$ be
    reduced skew bracoids. Now $(G_1,N,\odot_1)\cong (G_2,N,\odot_2)$
    if and only if $\lambda_{\odot_1}(G_1)$ is conjugate to
    $\lambda_{\odot_2}(G_2)$ by an element of $\Aut(N)$.
\end{proposition}
Thus, by Corollary \ref{conj_isom}, the number of equivalence classes
of skew bracoids with fixed groups $G$ and $N$, up to isomorphism, is 
the number of conjugacy classes of transitive subgroups $M$
 of $\Hol(N)$ with $M \cong G$.

\textit{Almost classical} skew bracoids were introduced in \cite{ML24}. 
\begin{definition}
    Let $(G,N,\odot)$ be a skew bracoid and
    $S=\Stab_{\odot}(1_N)$. We say that $(G,N,\odot)$ is
    \textit{almost classical} if $S$ has a normal complement $H$ in
    $G$ such that the sub-skew bracoid $(H,N,\odot|_H)$ is 
essentially a skew brace that is trivial
(that is, the operations on $H$ and $N$ coincide).
\end{definition}
Therefore, if $(G,N,\odot)$ is an almost classical skew bracoid, then
$G=H \rtimes S$ for some $H \unlhd G$ such that $(H,N,\odot|_H)$ is 
essentially a skew brace that is trivial.
The methods of \cite{CKMLT24} can now be used to show 
that $(G,N,\odot)$ yields a solution to the ST-YBE.

By \cite[Proposition 4.3]{ML24}, $(G,N,\odot)$ is an
almost classical skew bracoid if and only if $\lambda_{\odot}(G)=N
\rtimes A \subseteq \Hol(N)$ for some $A \leq \Aut(N)$. Suppose 
that $L/K$ is a separable extension with Galois groups $J=\Gal(E/K)$
and $J'=\Gal(E/L)$, and there is an isomorphism $\phi:J \to
\lambda_{\odot}(G)$ such that $\phi(J')=A$. Now $L/K$ is almost
classically Galois, admitting the almost classically Galois structure
given by the Hopf algebra $E[N^{opp}]^J$.

\section{The Algorithms}\label{alg}
Let $n$ be a positive integer. We now outline our algorithm to
classify the Hopf--Galois 
structures on separable extensions
and skew bracoids of degree $n$.

\begin{enumerate}[label=Step \arabic*:]
    \item For each group $N$ of order $n$, construct the sequence 
      $\texttt{t}_N$ of transitive subgroups of $\Hol(N)$ up to conjugacy.
Let $\texttt{T}_n$ be the concatenation of these sequences.
\item 
    Partition $\texttt{T}_n$ into equivalence classes: if $M_1,M_2 \in
      \texttt{T}_n$, then $M_1 \sim M_2$ if and only if there is an
      isomorphism $\phi:M_1 \to M_2$ such that
      $\phi(\Stab_{M_1}(1_{N_1}))=\Stab_{M_2}(1_{N_2})$
      for groups $N_1$ and $N_2$ of order $n$.

    \item For each equivalence class obtained in Step 2, compute the
      number of associated Hopf--Galois structures and skew bracoids,
      the number of those arising from Galois extensions / skew
      braces, the number of almost classical(ly Galois) structures,
      and the number of associated Hopf--Galois structures which admit
      a bijective correspondence.
\end{enumerate}

We can readily modify this algorithm as follows to obtain 
one which simply {\it counts} the number of structures of degree $n$.  
In Step 1, we compute $\texttt{t}_N$ for one group $N$
of order $n$.  Since the subgroups need not be sorted into 
equivalence classes, we skip Step 2.  
The input to Step 3 is $\texttt{t}_N$.
We now apply this modified algorithm to each group $N$ of order $n$.

We first consider how to recover the action needed to construct
the relevant structures. We then discuss in more detail the steps of the
algorithm and how it can be parallelised.

\subsection{Recovering the action on $N$}\label{action-recovery}
Let \mbox{$f:N \hookrightarrow S_n$} be 
the natural embedding of a group $N$ of order $n$ as a permutation group. 
{\sc Magma} 
constructs $\Hol(N)$ as the normaliser of $f(N)$ in $S_n$, and so 
the transitive subgroups of $\Hol(N)$
act on $\{1,\ldots,n\}$, which is not necessarily the underlying set of $N$. 
In this section, we describe how to transport the action on the former set 
to the latter set so that we can construct the relevant Hopf--Galois 
structures and skew bracoids.

Denote by $\Hol(N)$ the group $\mathrm{Norm}_{\Perm(N)}(\lambda(N))$ and by $\Hol_n(N)$ the group $\mathrm{Norm}_{S_n}(f(N))$. Define the homomorphism $h:\Hol_n(N) \to \Stab_{\Hol_n(N)}(1)$ which maps the generators of $f(N)$ to $1_{\Hol_n(N)}$ and fixes the generators of $\Stab_{\Hol_n(N)}(1)$, the stabiliser of $1$.
If $\pi \in \Hol(N)$, then $\pi=\lambda_\eta\alpha$ for some $\eta \in N$ and 
$\alpha \in \Aut(N)$. 
Let $\Phi:\Hol(N) \to \Hol_n(N)$ be an isomorphism 
such that
\begin{align*}
    \Phi(\lambda_{\eta})=f(\eta)\ && \text{and} && \Phi(\alpha) \in \Stab_{\Hol_n(N)}(1)
\end{align*}
for all $\eta \in N$ and $\alpha \in \Stab_{\Hol(N)}(1_N)=\Aut(N)$.

Note that $\alpha\lambda_\eta=\lambda_{\alpha(\eta)}\alpha$, and hence
\[(\lambda_{\eta}\alpha)(\lambda_{\mu})(\alpha^{-1})=\lambda_{\eta\alpha(\mu)}\]
for all $\eta,\mu \in N$ and $\alpha \in \Aut(N)$. 
If $x \in \Hol_n(N)$ satisfies 
$$x=\Phi(\lambda_\eta\alpha)=f(\eta)\Phi(\alpha)=f(\eta)h(x),$$ then 
\[\lambda_{\eta\alpha(\mu)}=\Phi^{-1}(x)\lambda_\mu\Phi^{-1}(h(x))^{-1}\]
so
\[f(\eta\alpha(\mu))=\Phi(\lambda_{\eta\alpha(\mu)})=xf(\mu)h(x)^{-1}.\]
Therefore 
\[(\eta,\alpha)\cdot \mu=\eta\alpha(\mu)=f^{-1}(xf(\mu)h(x)^{-1}).\]
Thus we may recover the action of $\Hol_n(N)$ on $N$ by
\begin{equation}\label{action}
    x \cdot \mu =f^{-1}(xf(\mu)h(x)^{-1}).
\end{equation}
While $\Phi$ depends on the choice of bijection between $N$ and 
$\{1,\ldots,n\}$, this dependence does not appear in the final construction.

For the rest of this section, let $M$ denote a transitive subgroup of 
$\Hol_n(N)$. We write $1$ for the corresponding element of 
$\{1,\ldots,n\}$, and $1_N$ for the identity element of $N$. 
Note that $\Stab_M(1)=\Stab_M(1_N)$.

\subsubsection{The action for Hopf--Galois structures}
Let $L/K$ be a separable extension of degree $n$ and 
let $E$ be its Galois closure. 
Define $J=\Gal(E/K)$ and $J'=\Gal(E/L)$, and suppose there is 
an isomorphism $\phi: J\to M$ such that $\phi(J')=\Stab_M(1_N)$. 
Now let $\overline{\phi}:J/J' \to N$ be the map given by $\overline{\phi}(gJ')=\phi(g)\cdot(1_N)$ for all $g \in J$, and identify $N$ with the image of the map $\alpha:N \hookrightarrow \Perm(J/J')$ given by $\alpha(\eta)=\overline{\phi}\eta\overline{\phi}^{-1}$ for all $\eta \in N$. As explained in Section \ref{HGS}, the conjugation action of $J$ on $\alpha(N)$ via $\lambda:J \to \Perm(J/J')$ 
gives $H=E[N]^J$, which can in turn be used to give the corresponding Hopf--Galois structure on $L/K$.

Note that this construction is the bijection given in 
\cite[Proposition 1]{Byo96}.
\begin{remark}
As discussed in \cite[\S 5]{MLT24}, we do not require that 
$E$ is the Galois closure of $L/K$, but only that $E/K$ is a 
Galois extension containing $L/K$. 
For simplicity, we present only the Galois closure case here. 
\end{remark}
      
\subsubsection{The action for skew bracoids}
Let $G$ be a group such that there is a homomorphism 
$\delta: G \to \Hol_n(N)$ with $\delta(G)=M$. Given $\cdot$ as above, 
we define an action $\odot$ of $G$ on $N$ by
\[g \odot \mu = \delta(g) \cdot \mu \;\; \forall g \in G, \forall \mu \in N.\]
Therefore, replacing $x$ in (\ref{action}) by $\delta(g)$, we obtain
\[g \odot \mu = f^{-1}(\delta(g)f(\mu)h(\delta(g))^{-1}) \;\; \forall g \in G, \forall \mu \in N.\]
Clearly, 
\[g\odot(\mu+\eta)=(g\odot\mu)-(g\odot1_N)+(g\odot\eta)\]
for all $g \in G$ and $\eta,\mu \in N$. 
Hence 
$(G,N,\odot)$ is a skew bracoid. If $\delta$ is bijective, 
then the corresponding skew bracoid is reduced.

From now on, we identify $\Hol_n(N)$ with $\Hol(N)$. 

\subsection{Realising Step 2}
In this step, we sort $T_n$ into equivalence classes. 
As mentioned in the outline of our algorithm, 
transitive subgroups $M_1 \leq \Hol(N_1)$ and 
$M_2 \leq \Hol(N_2)$ are \textit{equivalent} if $M_1$ is isomorphic 
to $M_2$ via an isomorphism sending $M_1' = \Stab_{M_1}(1_{N_1})$ to 
$M_2' = \Stab_{M_2}(1_{N_2})$, so they are permutation isomorphic. 
Recall that two equivalent transitive 
subgroups correspond to Hopf--Galois structures on the same separable 
extension, and also 
correspond to skew bracoids $(G,N_1,\odot_1)$ and $(G,N_2,\odot_2)$ with 
$\Stab_{\odot_1}(1_{N_1})\cong \Stab_{\odot_2}(1_{N_2})$. 
We first partition the groups into bins corresponding to isomorphism classes
as follows.
\begin{enumerate}
    \item If a group $M$ has a unique identifier in either the 
{\sc SmallGroups} library \cite{SmallGroups}, or the library of 
transitive groups \cite{Deg48}, then we use this identifier to place it 
in a bin. The bins resulting from identification of transitive groups are 
merged by deciding whether representatives are abstractly isomorphic. 
These identifiers are available in {\sc Magma} for many groups of 
order at most 1000, and for transitive groups of degree at most~30.
    \item If $M$ is a $p$-group, then we use its standard presentation 
\cite{OBrien94} to place it in a bin. Two $p$-groups have the same 
standard presentation if and only if they are isomorphic.
    \item Otherwise, we use the intrinsic isomorphism machinery
    available in {\sc Magma} to place $M$ in a bin. 
\end{enumerate}
Assume that the groups are now sorted into equivalence classes up to 
abstract isomorphism. Choose groups $M_1$ and $M_2$ from the same bin. 
{\sc Magma} provides an isomorphism $f:M_1 \to M_2$. But, even if $M_1$ 
and $M_2$ are permutation isomorphic, it may be that $f$ does 
not map $M_1'$ to $M_2'$. In theory, we could decide whether there 
exists $\psi \in \Aut(M_2)$ such that $(\psi \circ f)(M_1')=M_2'$, 
since this is equivalent to running over all isomorphisms between 
$M_1$ and $M_2$. 

As discussed in \cite[\S 2]{CS21}, a more efficient approach is 
the following:
compute the set of images, \texttt{Gprime}, of $M_2'$ under $\Aut(M_2)$ 
and decide if $f(M_1')$ lies in \texttt{Gprime}. If
so, then $f(M_1')=\psi(M_2')$ for some $\psi \in \Aut(M_2)$, and 
$(\psi^{-1} \circ f) : M_1 \to M_2$ is the required 
permutation 
isomorphism.

Given a transitive subgroup $M$ of $\Hol(N)$, Crespo and Salguero \cite{CS21}
used a function provided by Derek Holt which exploits 
a permutation representation of $\Aut(M)$ to construct \texttt{Gprime}. 

We have developed a very efficient procedure to construct \texttt{Gprime}
which incorporates the following improvements.
\begin{itemize}
    \item If $M$ is a finite soluble  or $p$-group, then $\Aut(M)$ is computed 
    using algorithms of \cite{Howden} and \cite{ELO} designed for such groups;
    these are typically faster than the equivalent
    algorithm \cite{CannonHolt} for arbitrary permutation groups.
\item We choose a `small' generating set for the permutation
    representation of $\Aut(M)$.
\item We apply the {\sc Magma} intrinsic \texttt{CanonicalInvariant} to
    each image of $M_2'$ under an automorphism. Based on an algorithm
    of Hulpke and Linton \cite{canonical}, this function takes as input a
    permutation group and outputs a canonical invariant of the group:
    a particular generating set stored as integer sequences. 
    Two permutation groups are equal if and only if they
    have the same canonical invariant. 
\item We store \texttt{Gprime} as an {\it indexed set} of canonical invariants.
    This significantly reduces the time to decide 
    membership of $f(M_1')$ in \texttt{Gprime}.
\end{itemize}

Our implementation returns both 
    \texttt{Gprime} and $|\Aut(M,M')|$, the number of
    $M$-automorphisms fixing $M'=\Stab_M(1_N)$. The latter 
    is used in Step 3 to compute the number of associated
    Hopf--Galois structures.

\subsection{Realising Step 3}
Recall that the number of (equivalence classes of) skew bracoids $(G, N, \odot)$ is 
precisely the number of conjugacy classes of transitive subgroups $M$ of 
$\Hol(N)$ permutation isomorphic to $G$.
Those which are essentially skew braces arise from regular subgroups of 
$\Hol(N)$. To count the number of Hopf--Galois structures of type $N$ 
admitted by a Galois extension of degree $n$ with Galois group $J$, 
we count \textit{all} regular subgroups $M$ of $\Hol(N)$ 
isomorphic to $J$ (so we sum the number of subgroups of $\Hol(N)$ in 
the equivalence class of $J$ computed in Step 2 for each group $N$, 
accounting for the sizes of the conjugacy classes), and multiply this count by 
$|\Aut(J)|/|\Aut(N)|$. (In practice, this factor is computed 
as $|\Aut(M)|/|\Aut(N)|$.) 

In contrast, by Lemma \ref{Byott_num_HGS}, the number 
of Hopf--Galois structures of type $N$ admitted by a \textit{separable} (but not necessarily normal) extension with pair of groups $(J,J')$ is the number of 
transitive subgroups $M$ of $\Hol(N)$ isomorphic to $J$ via an isomorphism 
taking $M'=\Stab_M(1_N)$ to $J'$, for a fixed $N$, 
multiplied by $|\Aut(J,J')|/|\Aut(N)|$. 
(This factor is 
computed as $|\Aut(M,M')|/|\Aut(N)|$;
recall $|\Aut(M,M')|$ was computed in Step 2). 

Let $M \leq \Hol(N)$ be transitive and isomorphic to $J$. 
To count the number of Hopf--Galois structures admitting a bijective
correspondence, we do the following:
\begin{itemize}
\item 
Compute the number of intermediate fields
of $L/K$; do this, using the classical Galois correspondence, by
counting the number of subgroups of $M$ containing~$M'$.

\item 
Compute the number of Hopf subalgebras of $E[N]^J$;
do this by counting how many subgroups $H$ of $N$ have the property 
that $M \leq \mathrm{Norm}_{\Hol(N)}(H)$.
\end{itemize}
If these numbers agree, then the
corresponding Hopf--Galois structure admits a bijective correspondence.

For each transitive subgroup $M$ of $\Hol(N)$, 
the corresponding structure is almost
classical(ly Galois) if and only if $\mathrm{Cent}_{S_n}(N) \leq G$. 
(For justification, consider the discussion at the
end of Section \ref{HGS} and the observation that $N^{\text{opp}}$ is
precisely the centraliser of $N$ in $S_n$.)

\subsection{Exploiting parallelisation}
Our algorithms can readily run in parallel.
We use existing features of {\sc Magma} to realise this:
it uses a \textit{manager--worker} model, where the
`manager' organises and delegates (a sequence of) tasks to 
`workers'. We use parallelisation for the following tasks.

\begin{enumerate}[label=(\roman*)]
    \item\label{step1} Each worker is assigned a different
      group $N$ of order $n$ and is tasked with computing the
      conjugacy classes of transitive subgroups of $\Hol(N)$. 

    \item Each worker is assigned one of
      the transitive groups $M$ computed in \ref{step1}.
      It constructs \texttt{Gprime} and $|\Aut(M,M')|$.
      It also decides whether $G$ is regular, or corresponds to either 
      an almost classical(ly Galois) structure
      or a Hopf--Galois structure admitting a bijective correspondence.

    \item Each worker is assigned the isomorphism problem described in 
      Step 2 for those subgroups of a given order. 

    \item Each worker is assigned an
      equivalence class of transitive subgroups, as computed in Step 2, 
      and enumerates the corresponding Hopf--Galois
      structures and skew bracoids. (This is effectively Step 3 of the
      algorithm, but $|\Aut(M,M')|$ is already known.)
\end{enumerate}

\section{Enumerations and structures}\label{results}
In this section, we present enumerations and structure information
for Hopf--Galois structures and skew bracoids of small degree.
Our results were obtained using the algorithms given in Section \ref{alg}. 

\subsection{Enumerations}\label{enumi-results}
We enumerated the Hopf--Galois structures on separable extensions
and skew bracoids up to degree 200, excluding 
those of degree 64, 96, 128, 144, 160, 162 and 192. 

Tables \ref{tab:enum1}--\ref{tab:enum4}
summarise the enumeration results. 
The first column lists the degree
$n$ of the extension. The second lists the number of groups of
order $n$ up to isomorphism. The third 
lists both the total number of Hopf--Galois structures and skew
bracoids of that degree. The fourth 
lists the number of Hopf--Galois structures on Galois extensions
and the number of skew braces. The fifth lists the 
number of each structure which are almost classically Galois. The sixth 
lists the total number of Hopf--Galois structures giving bijective
correspondence.

The `Hopf--Galois part' of Table \ref{tab:enum1}
recovers some of the results of \cite{CS20, CS21}.
Our results include the first complete enumerations for degrees 
32, 48, 72, 80 and 81.  In particular, the number of 
Hopf--Galois structures on Galois extensions of degree 32 is 61482704. 

\subsection{Structures}
We used the classification algorithm  to construct the  
Hopf--Galois structures  and skew bracoids of degree up to 100, excluding 
those of degree 32, 48, 64, 80, 81 and 96.  
The isomorphisms required in Step 2 dictate that this algorithm is
markedly more expensive than the enumeration alternative. 

For a given integer $n$, our implementation outputs two sequences. 
The first records the transitive subgroups of 
$\Hol(N)$ for each group $N$ of order $n$; the second records 
the equivalence classes of these transitive subgroups and 
the corresponding number of structures.

\ifversion

\begin{center}
\begin{table}[htp]
	\begin{tabular}{|l|l|l|l|l|l|l|l|l|}
		\hline \multirow{2}{3em}{Degree} &
                \multirow{2}{3em}{Types} & \multicolumn{2}{c|}{Total}
                & \multicolumn{2}{c|}{Regular} &
 \multicolumn{2}{c|}{Almost classical} & \multirow{2}{2em}{\centering
 \#BC HGS}\\ \cline{3-8} & & \#HGS & \#Sbracoids &
 \#Gal &\#Sbraces & \#HGS & \#Sbracoids &\\ \hline 
2 & 1 & 1 & 1 & 1 & 1 & 1 & 1 & 1 \\
3 & 1 & 2 & 2 & 1 & 1 & 2 & 2 & 2 \\
4 & 2 & 10 & 8 & 6 & 4 & 6 & 6 & 7 \\
5 & 1 & 3 & 3 & 1 & 1 & 3 & 3 & 3 \\
6 & 2 & 15 & 12 & 8 & 6 & 7 & 6 & 9 \\
7 & 1 & 4 & 4 & 1 & 1 & 4 & 4 & 4 \\
8 & 5 & 348 & 148 & 190 & 47 & 74 & 47 & 147 \\
9 & 2 & 38 & 23 & 12 & 4 & 26 & 20 & 28 \\
10 & 2 & 27 & 20 & 10 & 6 & 11 & 9 & 17 \\
11 & 1 & 4 & 4 & 1 & 1 & 4 & 4 & 4 \\
12 & 5 & 249 & 134 & 102 & 38 & 56 & 46 & 81 \\
13 & 1 & 6 & 6 & 1 & 1 & 6 & 6 & 6 \\
14 & 2 & 32 & 24 & 12 & 6 & 14 & 12 & 19 \\
15 & 1 & 8 & 8 & 1 & 1 & 8 & 8 & 8 \\
16 & 14 & 49913 & 9739 & 25168 & 1605 & 2636 & 815 & 9331 \\
17 & 1 & 5 & 5 & 1 & 1 & 5 & 5 & 5 \\
18 & 5 & 881 & 333 & 289 & 49 & 123 & 89 & 253 \\
19 & 1 & 6 & 6 & 1 & 1 & 6 & 6 & 6 \\
20 & 5 & 434 & 203 & 166 & 43 & 79 & 62 & 156 \\
21 & 2 & 78 & 36 & 28 & 8 & 22 & 18 & 46 \\
22 & 2 & 36 & 24 & 16 & 6 & 14 & 12 & 19 \\
23 & 1 & 4 & 4 & 1 & 1 & 4 & 4 & 4 \\
24 & 15 & 14908 & 4752 & 5618 & 855 & 844 & 504 & 2682 \\
25 & 2 & 106 & 58 & 30 & 4 & 70 & 54 & 74 \\
26 & 2 & 58 & 40 & 18 & 6 & 22 & 18 & 35 \\
27 & 5 & 6699 & 739 & 4329 & 101 & 766 & 283 & 1100 \\
28 & 4 & 388 & 202 & 128 & 29 & 84 & 72 & 143 \\
29 & 1 & 6 & 6 & 1 & 1 & 6 & 6 & 6 \\
30 & 4 & 479 & 304 & 80 & 36 & 99 & 72 & 197 \\
31 & 1 & 8 & 8 & 1 & 1 & 8 & 8 & 8 \\
32 & 51 & 151056415 & 26057416 & 61482704 & 1223061 & 709575 & 50254 & 10338058 \\
33 & 1 & 10 & 10 & 1 & 1 & 10 & 10 & 10 \\
34 & 2 & 59 & 36 & 22 & 6 & 19 & 15 & 33 \\
35 & 1 & 16 & 16 & 1 & 1 & 16 & 16 & 16 \\
36 & 14 & 16512 & 4159 & 5980 & 400 & 1099 & 753 & 2474 \\
37 & 1 & 9 & 9 & 1 & 1 & 9 & 9 & 9 \\
38 & 2 & 57 & 36 & 24 & 6 & 21 & 18 & 29 \\
39 & 2 & 133 & 55 & 46 & 8 & 34 & 28 & 77 \\
40 & 14 & 29534 & 8873 & 8556 & 944 & 1486 & 831 & 5931 \\
41 & 1 & 8 & 8 & 1 & 1 & 8 & 8 & 8 \\
42 & 6 & 1041 & 484 & 374 & 78 & 148 & 112 & 329 \\
43 & 1 & 8 & 8 & 1 & 1 & 8 & 8 & 8 \\
44 & 4 & 466 & 200 & 184 & 29 & 82 & 70 & 141 \\
45 & 2 & 166 & 115 & 12 & 4 & 126 & 104 & 132 \\
46 & 2 & 48 & 24 & 28 & 6 & 14 & 12 & 19 \\
47 & 1 & 4 & 4 & 1 & 1 & 4 & 4 & 4 \\
48 & 52 & 4490340 & 874252 & 1314000 & 66209 & 48347 & 16595 & 402022 \\
49 & 2 & 200 & 97 & 56 & 4 & 122 & 92 & 128 \\
50 & 5 & 3430 & 978 & 969 & 51 & 339 & 235 & 865 \\
\hline
\end{tabular}
\caption{Enumeration results (degrees 2--50)}
\label{tab:enum1}
\end{table}
\end{center}

\begin{center}
\begin{table}[htp]
	\begin{tabular}{|l|l|l|l|l|l|l|l|l|}
		\hline \multirow{2}{3em}{Degree} &
                \multirow{2}{3em}{Types} & \multicolumn{2}{c|}{Total}
                & \multicolumn{2}{c|}{Regular} &
 \multicolumn{2}{c|}{Almost classical} & \multirow{2}{2em}{\centering
 \#BC HGS}\\ \cline{3-8} & & \#HGS & \#Sbracoids &
 \#Gal &\#Sbraces & \#HGS & \#Sbracoids &\\ \hline 
51 & 1 & 14 & 14 & 1 & 1 & 14 & 14 & 14 \\
52 & 5 & 1023 & 409 & 374 & 43 & 161 & 127 & 343 \\
53 & 1 & 6 & 6 & 1 & 1 & 6 & 6 & 6 \\
54 & 15 & 234466 & 16017 & 144467 & 1028 & 3071 & 1953 & 9927 \\
55 & 2 & 192 & 54 & 88 & 12 & 32 & 24 & 94 \\
56 & 13 & 32721 & 9227 & 10010 & 815 & 1620 & 968 & 5747 \\
57 & 2 & 169 & 61 & 64 & 8 & 35 & 27 & 93 \\
58 & 2 & 74 & 40 & 34 & 6 & 22 & 18 & 35 \\
59 & 1 & 4 & 4 & 1 & 1 & 4 & 4 & 4 \\
60 & 13 & 13457 & 4621 & 3128 & 418 & 947 & 668 & 2529 \\
61 & 1 & 12 & 12 & 1 & 1 & 12 & 12 & 12 \\
62 & 2 & 82 & 48 & 36 & 6 & 28 & 24 & 39 \\
63 & 4 & 1875 & 501 & 504 & 47 & 335 & 207 & 749 \\
64 & 267 & ? & ? & ? & ? & ? & ? & ? \\
65 & 1 & 30 & 30 & 1 & 1 & 30 & 30 & 30 \\
66 & 4 & 608 & 352 & 128 & 36 & 118 & 90 & 211 \\
67 & 1 & 8 & 8 & 1 & 1 & 8 & 8 & 8 \\
68 & 5 & 1162 & 391 & 478 & 43 & 145 & 108 & 352 \\
69 & 1 & 10 & 10 & 1 & 1 & 10 & 10 & 10 \\
70 & 4 & 1012 & 608 & 120 & 36 & 198 & 144 & 411 \\
71 & 1 & 8 & 8 & 1 & 1 & 8 & 8 & 8 \\
72 & 50 & 2004057 & 329821 & 646560 & 17790 & 23474 & 13060 & 135087 \\
73 & 1 & 12 & 12 & 1 & 1 & 12 & 12 & 12 \\
74 & 2 & 105 & 60 & 42 & 6 & 33 & 27 & 53 \\
75 & 3 & 1795 & 357 & 597 & 14 & 290 & 230 & 330 \\
76 & 4 & 763 & 304 & 296 & 29 & 127 & 109 & 220 \\
77 & 1 & 20 & 20 & 1 & 1 & 20 & 20 & 20 \\
78 & 6 & 1957 & 828 & 650 & 78 & 244 & 177 & 637 \\
79 & 1 & 8 & 8 & 1 & 1 & 8 & 8 & 8 \\
80 & 52 & 9219006 & 1723150 & 2123494 & 74120 & 93103 & 30020 & 953767 \\
81 & 15 & 15897515 & 68549 & 13781853 & 8436 & 137484 & 7470 & 389829 \\
82 & 2 & 106 & 56 & 46 & 6 & 30 & 24 & 51 \\
83 & 1 & 4 & 4 & 1 & 1 & 4 & 4 & 4 \\
84 & 15 & 21790 & 6371 & 6232 & 606 & 1271 & 925 & 3530 \\
85 & 1 & 29 & 29 & 1 & 1 & 29 & 29 & 29 \\
86 & 2 & 94 & 48 & 48 & 6 & 28 & 24 & 39 \\
87 & 1 & 16 & 16 & 1 & 1 & 16 & 16 & 16 \\
88 & 12 & 41020 & 9120 & 14584 & 800 & 1568 & 934 & 5683 \\
89 & 1 & 8 & 8 & 1 & 1 & 8 & 8 & 8 \\
90 & 10 & 30167 & 10256 & 2890 & 294 & 2165 & 1365 & 6611 \\
91 & 1 & 48 & 48 & 1 & 1 & 48 & 48 & 48 \\
92 & 4 & 706 & 200 & 352 & 29 & 82 & 70 & 141 \\
93 & 2 & 246 & 72 & 100 & 8 & 44 & 36 & 130 \\
94 & 2 & 72 & 24 & 52 & 6 & 14 & 12 & 19 \\
95 & 1 & 24 & 24 & 1 & 1 & 24 & 24 & 24 \\
96 & 231 & ? & ? & ? & ? & ? & ? & ? \\
97 & 1 & 12 & 12 & 1 & 1 & 12 & 12 & 12 \\
98 & 5 & 6824 & 1541 & 2265 & 53 & 576 & 413 & 1350 \\
99 & 2 & 202 & 136 & 12 & 4 & 150 & 122 & 158 \\
100 & 16 & 119542 & 15397 & 55732 & 711 & 3200 & 2165 & 10777 \\
\hline
\end{tabular}
\caption{Enumeration results (degrees 51--100)}
\label{tab:enum2}
\end{table}
\end{center}

\begin{center}
\begin{table}[htp]
	\begin{tabular}{|l|l|l|l|l|l|l|l|l|}
		\hline \multirow{2}{3em}{Degree} &
                \multirow{2}{3em}{Types} & \multicolumn{2}{c|}{Total}
                & \multicolumn{2}{c|}{Regular} &
 \multicolumn{2}{c|}{Almost classical} & \multirow{2}{2em}{\centering
 \#BC HGS}\\ \cline{3-8} & & \#HGS & \#Sbracoids &
 \#Gal &\#Sbraces & \#HGS & \#Sbracoids &\\ \hline 
101 & 1 & 9 & 9 & 1 & 1 & 9 & 9 & 9 \\
102 & 4 & 1039 & 560 & 176 & 36 & 179 & 126 & 389 \\
103 & 1 & 8 & 8 & 1 & 1 & 8 & 8 & 8 \\
104 & 14 & 69723 & 17825 & 19804 & 944 & 3015 & 1695 & 12632 \\
105 & 2 & 296 & 158 & 28 & 8 & 102 & 86 & 188 \\
106 & 2 & 98 & 40 & 58 & 6 & 22 & 18 & 35 \\
107 & 1 & 4 & 4 & 1 & 1 & 4 & 4 & 4 \\
108 & 45 & 4622126 & 237578 & 2496986 & 11223 & 35398 & 18708 & 119695 \\
109 & 1 & 12 & 12 & 1 & 1 & 12 & 12 & 12 \\
110 & 6 & 2233 & 776 & 862 & 94 & 236 & 166 & 647 \\
111 & 2 & 300 & 93 & 118 & 8 & 54 & 42 & 160 \\
112 & 43 & 10425717 & 1739295 & 2608168 & 65485 & 95174 & 32188 & 884165 \\
113 & 1 & 10 & 10 & 1 & 1 & 10 & 10 & 10 \\
114 & 6 & 2233 & 774 & 926 & 78 & 229 & 171 & 597 \\
115 & 1 & 16 & 16 & 1 & 1 & 16 & 16 & 16 \\
116 & 5 & 1628 & 406 & 790 & 43 & 158 & 124 & 372 \\
117 & 4 & 3223 & 793 & 864 & 47 & 537 & 341 & 1253 \\
118 & 2 & 84 & 24 & 64 & 6 & 14 & 12 & 19 \\
119 & 1 & 28 & 28 & 1 & 1 & 28 & 28 & 28 \\
120 & 47 & 2223820 & 571879 & 359042 & 22711 & 27975 & 15464 & 217324 \\
121 & 2 & 332 & 127 & 132 & 4 & 172 & 122 & 182 \\
122 & 2 & 152 & 80 & 66 & 6 & 44 & 36 & 71 \\
123 & 1 & 22 & 22 & 1 & 1 & 22 & 22 & 22 \\
124 & 4 & 1136 & 404 & 464 & 29 & 168 & 144 & 295 \\
125 & 5 & 137732 & 2738 & 117525 & 213 & 5222 & 1299 & 7102 \\
126 & 16 & 71266 & 17700 & 17082 & 990 & 3402 & 2207 & 10855 \\
127 & 1 & 12 & 12 & 1 & 1 & 12 & 12 & 12 \\
128 & 2328 & ? & ? & ? & ? & ? & ? & ? \\
129 & 2 & 306 & 72 & 136 & 8 & 44 & 36 & 154 \\
130 & 4 & 2098 & 1248 & 180 & 36 & 382 & 270 & 915 \\
131 & 1 & 8 & 8 & 1 & 1 & 8 & 8 & 8 \\
132 & 10 & 16716 & 4584 & 4538 & 324 & 948 & 714 & 2260 \\
133 & 1 & 50 & 50 & 1 & 1 & 50 & 50 & 50 \\
134 & 2 & 118 & 48 & 72 & 6 & 28 & 24 & 39 \\
135 & 5 & 23517 & 3781 & 4329 & 101 & 3790 & 1757 & 5000 \\
136 & 15 & 80962 & 17496 & 26254 & 986 & 2844 & 1529 & 13074 \\
137 & 1 & 8 & 8 & 1 & 1 & 8 & 8 & 8 \\
138 & 4 & 788 & 352 & 224 & 36 & 118 & 90 & 211 \\
139 & 1 & 8 & 8 & 1 & 1 & 8 & 8 & 8 \\
140 & 11 & 28694 & 8802 & 5108 & 395 & 1716 & 1194 & 5221 \\
141 & 1 & 10 & 10 & 1 & 1 & 10 & 10 & 10 \\
142 & 2 & 122 & 48 & 76 & 6 & 28 & 24 & 39 \\
143 & 1 & 32 & 32 & 1 & 1 & 32 & 32 & 32 \\
144 & 197 & ? & ? & ? & ? & ? & ? & ? \\
145 & 1 & 30 & 30 & 1 & 1 & 30 & 30 & 30 \\
146 & 2 & 171 & 84 & 78 & 6 & 45 & 36 & 77 \\
147 & 6 & 42964 & 2991 & 23654 & 123 & 987 & 725 & 3257 \\
148 & 5 & 2220 & 615 & 998 & 43 & 243 & 192 & 562 \\
149 & 1 & 6 & 6 & 1 & 1 & 6 & 6 & 6 \\
150 & 13 & 83527 & 20051 & 19903 & 401 & 3965 & 2571 & 12865 \\
\hline
\end{tabular}
\caption{Enumeration results (degrees 101--150)}
\label{tab:enum3}
\end{table}
\end{center}

\begin{center}
\begin{table}[htp]
	\begin{tabular}{|l|l|l|l|l|l|l|l|l|}
		\hline \multirow{2}{3em}{Degree} &
                \multirow{2}{3em}{Types} & \multicolumn{2}{c|}{Total}
                & \multicolumn{2}{c|}{Regular} &
 \multicolumn{2}{c|}{Almost classical} & \multirow{2}{2em}{\centering
 \#BC HGS}\\ \cline{3-8} & & \#HGS & \#Sbracoids &
 \#Gal &\#Sbraces & \#HGS & \#Sbracoids &\\ \hline 
151 & 1 & 12 & 12 & 1 & 1 & 12 & 12 & 12 \\
152 & 12 & 68147 & 13778 & 24104 & 800 & 2408 & 1445 & 8849 \\
153 & 2 & 300 & 213 & 12 & 4 & 232 & 194 & 242 \\
154 & 4 & 1288 & 704 & 192 & 36 & 236 & 180 & 441 \\
155 & 2 & 494 & 108 & 228 & 12 & 64 & 48 & 232 \\
156 & 18 & 48845 & 12425 & 14070 & 782 & 2403 & 1650 & 8052 \\
157 & 1 & 12 & 12 & 1 & 1 & 12 & 12 & 12 \\
158 & 2 & 130 & 48 & 84 & 6 & 28 & 24 & 39 \\
159 & 1 & 16 & 16 & 1 & 1 & 16 & 16 & 16 \\
160 & 238 & ? & ? & ? & ? & ? & ? & ? \\
161 & 1 & 20 & 20 & 1 & 1 & 20 & 20 & 20 \\
162 & 55 & ? & ? & ? & ? & ? & ? & ? \\
163 & 1 & 10 & 10 & 1 & 1 & 10 & 10 & 10 \\
164 & 5 & 2358 & 594 & 1102 & 43 & 224 & 170 & 570 \\
165 & 2 & 439 & 149 & 88 & 12 & 94 & 74 & 231 \\
166 & 2 & 108 & 24 & 88 & 6 & 14 & 12 & 19 \\
167 & 1 & 4 & 4 & 1 & 1 & 4 & 4 & 4 \\
168 & 57 & 2993181 & 666155 & 656810 & 28505 & 33023 & 19363 & 240421 \\
169 & 2 & 565 & 236 & 182 & 4 & 313 & 229 & 325 \\
170 & 4 & 2279 & 1264 & 220 & 36 & 383 & 261 & 977 \\
171 & 5 & 5253 & 994 & 1941 & 80 & 579 & 337 & 1731 \\
172 & 4 & 1376 & 404 & 632 & 29 & 168 & 144 & 295 \\
173 & 1 & 6 & 6 & 1 & 1 & 6 & 6 & 6 \\
174 & 4 & 1258 & 608 & 272 & 36 & 198 & 144 & 403 \\
175 & 2 & 620 & 400 & 30 & 4 & 460 & 380 & 476 \\
176 & 42 & 13505976 & 1737666 & 3914032 & 65466 & 94602 & 31806 & 882938 \\
177 & 1 & 10 & 10 & 1 & 1 & 10 & 10 & 10 \\
178 & 2 & 154 & 56 & 94 & 6 & 30 & 24 & 51 \\
179 & 1 & 4 & 4 & 1 & 1 & 4 & 4 & 4 \\
180 & 37 & 1410782 & 223581 & 257602 & 5849 & 25332 & 15370 & 108626 \\
181 & 1 & 18 & 18 & 1 & 1 & 18 & 18 & 18 \\
182 & 4 & 2898 & 1824 & 216 & 36 & 594 & 432 & 1289 \\
183 & 2 & 446 & 110 & 190 & 8 & 68 & 56 & 228 \\
184 & 12 & 67048 & 9120 & 28864 & 800 & 1568 & 934 & 5683 \\
185 & 1 & 45 & 45 & 1 & 1 & 45 & 45 & 45 \\
186 & 6 & 3324 & 968 & 1478 & 78 & 296 & 224 & 801 \\
187 & 1 & 28 & 28 & 1 & 1 & 28 & 28 & 28 \\
188 & 4 & 1186 & 200 & 688 & 29 & 82 & 70 & 141 \\
189 & 13 & 862051 & 55655 & 316611 & 4560 & 16884 & 4584 & 78401 \\
190 & 4 & 1761 & 912 & 240 & 36 & 297 & 216 & 625 \\
191 & 1 & 8 & 8 & 1 & 1 & 8 & 8 & 8 \\
192 & 1543 & ? & ? & ? & ? & ? & ? & ? \\
193 & 1 & 14 & 14 & 1 & 1 & 14 & 14 & 14 \\
194 & 2 & 202 & 88 & 102 & 6 & 46 & 36 & 83 \\
195 & 2 & 573 & 303 & 46 & 8 & 198 & 168 & 369 \\
196 & 12 & 166477 & 18192 & 82216 & 389 & 4385 & 3158 & 12763 \\
197 & 1 & 9 & 9 & 1 & 1 & 9 & 9 & 9 \\
198 & 10 & 37378 & 11488 & 4624 & 294 & 2422 & 1614 & 6811 \\
199 & 1 & 12 & 12 & 1 & 1 & 12 & 12 & 12 \\
200 & 52 & 11405802 & 1322437 & 3983356 & 23471 & 81275 & 43266 & 683250 \\
\hline
\end{tabular}
\caption{Enumeration results (degrees 151--200)}
\label{tab:enum4}
\end{table}
\end{center}

\fi

In more detail, for a given $N$, we record the following data.
\begin{itemize}
\item The equivalent permutation groups (up to conjugacy) 
which are subgroups of $\Hol(N)$ and the total number of such subgroups; 
we identify which are regular.

\item Those subgroups which correspond to Hopf--Galois structures giving 
a bijective Hopf--Galois correspondence. 

\item Those subgroups which correspond to almost classical structures.

\item The total number of Hopf--Galois structures; the number 
which admit bijective correspondence; the number which 
are almost classically Galois; and the number of corresponding 
field extensions which are either Galois or non-normal.

\item The number of skew bracoids; the number which are almost classical;
and whether these are essentially skew braces.
\end{itemize}
Recall that equivalent permutation groups give rise to Hopf--Galois structures on the same separable extensions, and also to (equivalence classes of) skew bracoids with the same transitive group $G$ and isomorphic stabilisers.
 
Since the amount of data constructed for each group is large, 
we do not present it here, but instead refer 
to \cite{GITPAGE} for the detailed results. 
These are recorded in a format
permitting ready access and further study within {\sc Magma}.

\begin{example}
We illustrate the stored data by reference to the groups of order 4. 
Let $N_1$ and $N_2$  be the cyclic and elementary abelian groups 
of order 4 respectively. 
The stored data 
shows that $T_1 = \langle (1, 3, 2, 4) \rangle \leq \Hol(N_1)$ 
and $T_2 = \langle (1,3,2,4), (1,2)(3,4) \rangle \leq \Hol(N_2)$ are 
isomorphic regular permutation groups. The conjugacy classes of $T_1$ 
and $T_2$ have sizes one and three respectively. The following 
information is recorded about the corresponding Hopf--Galois structures and 
skew bracoids: every Galois extension  $L/K$ such that 
$\Gal(L/K)\cong T_1 \cong T_2$ admits  one Hopf--Galois structure of 
type $N_1$, corresponding to $T_1$ (which admits a bijective 
correspondence and is almost classically Galois), 
and one Hopf--Galois structure  of type $N_2$, corresponding to $T_2$ 
(which admits a bijective correspondence, but is not almost classically 
Galois). Suppose that $G$ is isomorphic to $T_i$. 
Up to isomorphism, there is one equivalence class of skew bracoids 
of the form $(G,N_1,\odot_1)$ and one of the form $(G,N_2,\odot_2)$ (given by 
transitive actions that can be recovered as discussed in Section 
\ref{action-recovery}). Both (classes of) skew bracoids are essentially 
skew braces, the former is also almost classical.
\end{example}

\subsection{The computations}
The computations were carried out using {\sc Magma} V2.29-3 
on a 3.1GHz machine running Ubuntu 24.04 with 128 processors and 
shared memory of 984GB. 
Using the manager-worker model in {\sc Magma}, we exploited
up to 100 of these.  The enumerations 
used 414 days of CPU time, the most expensive 
were for degrees 32, 80, 176 and 200.  The classifications -- restricted
to a significantly smaller set of degrees -- used  17.5 days of CPU time; 
that for degree 72 used 17 days. Detailed information on the individual 
computations is available  at \cite{GITPAGE}.

We excluded a group order when we could not complete the calculations for 
all of the groups of this order. The principal 
limitation in completing an order is typically constructing the 
conjugacy classes of all transitive subgroups of the related holomorphs;
the isomorphisms required pose an additional challenge to the 
structure algorithm.
But many of the groups from an excluded order can be processed readily 
using our implementations.

\section{
Hopf--Galois structures and skew bracoids of degree $2pq$}\label{ord2pq}

Let $p$ and $q$ be distinct odd primes with $p>q$. Let $N$ and $N'$ be
groups of order $2pq$ and let $G$ be a finite group. 
The pair $(G,N)$ is 
\textit{admissible} if $G$ can be embedded as a transitive subgroup of 
$\Hol(N)$.  
By the discussion in Section \ref{prelim}, 
the admissible pair $(G,N)$ corresponds to:
\begin{enumerate}
    \item the existence of a Hopf--Galois structure of type $N$ 
and group $G$, and
    \item the existence of a skew bracoid $(G,N,\odot)$. 
\end{enumerate}
In particular, the statement
\begin{quotation}
    \textit{
If $(G,N)$ is admissible, then $(G,N')$ is admissible.
}
\end{quotation}
may be interpreted in one of two ways:
\begin{center}
\begin{enumerate}
    \item \textit{Let $L/K$ be a separable extension of degree $|N|$
and Galois closure $E$
      such that $\Gal(E/K) \cong G$. If $L/K$ admits a Hopf--Galois
      structure of type $N$, then it also admits a Hopf--Galois
      structure of type $N'$.}
    \item \textit{If $(G,N,\odot)$ is a skew bracoid, then
      there is a transitive action $\odot'$ such that $(G,N',\odot')$
      is a skew bracoid.}
\end{enumerate}
\end{center}

Given an admissible pair $(G,N)$, we now investigate 
conditions for $(G,N')$ to be also admissible.
Our results 
are summarised in Figure \ref{fig:2pq-results}. 
An arrow from $(G,N)$ to $(G,N')$ signifies 
that if $(G,N)$ is admissible, then $(G,N')$ is admissible. 
The arrows labelled $2pq(p-1)$ 
apply if and only if $|G|$ divides $2pq(p-1)$. 
The groups $\C_2 \times (\C_p \rtimes \C_q)$ and $(\C_p \rtimes \C_{2q})$ exist only when $p \equiv 1 \bmod q$.
\begin{figure}
    \centering
    \[\begin{tikzcd}
	   & {(G,\D_{pq})} \\
	   {(G,\C_p \times \D_q)} && {(G,\C_q \times \D_p)} \\
	   & {(G,\C_{2pq})} \\
	   {(G,\C_2\times(\C_p \rtimes \C_q))} && {(G,\C_p \rtimes \C_{2q})}
	   \arrow[hook, from=2-1, to=1-2]
	   \arrow[hook', from=2-3, to=1-2]
	   \arrow[hook', from=3-2, to=2-1]
	   \arrow[hook, from=3-2, to=2-3]
	   \arrow["{2pq(p-1)}", hook, from=4-1, to=3-2]
	   \arrow["{2pq(p-1)}", <->, from=4-1, to=4-3]
      \arrow["{2pq(p-1)}"', hook', from=4-3, to=3-2]
    \end{tikzcd}\]
    \caption{Summary of results for degree $2pq$}
    \label{fig:2pq-results}
\end{figure}

\subsection{The groups of order $2pq$}
We first describe the groups of order $2pq$ and their 
automorphism groups. 
Since $2pq$ is squarefree, the following is a consequence 
of \cite[Lemmas 2.1 and 4.1]{AB20}.
Four (isomorphism types of) groups of order $2pq$
exist for all odd primes $p$ and $q$ where $p > q$. 
Subject to the convention that $g^h = hgh^{-1}$,  
these are: 
    \begin{align*} 
        \begin{split}
N_1&=\C_{2pq} = \langle x_1 \mid x_1^{2pq}=1 \rangle,\\ 
N_2&=\C_p\times \D_q = \langle x_2, r_2, s_2
            \mid x_2^p=r_2^q=s_2^2=1, r_2^{s_2}=r_2^{-1},
            x_2^{s_2}=x_2^{r_2}=x_2 \rangle,\\ 
N_3&=\C_q \times
            \D_p= \langle x_3, r_3, s_3 \mid x_3^q=r_3^p=s_3^2=1,
            r_3^{s_3}=r_3^{-1}, x_3^{s_3}=x_3^{r_3}=x_3
            \rangle,\\ 
N_4&=\D_{pq} = \langle r_4,s_4 \mid
            r_4^{pq}=s_4^2=1, r_4^{s_4}=r_4^{-1} \rangle.
        \end{split}
    \end{align*}
If $p \equiv 1 \bmod q$, then there are two additional groups,
where 
$k$ has order $q$ modulo~$p$: 
    \begin{align*} 
        \begin{split}
            N_5&=\C_2 \times (\C_p \rtimes \C_q) = \langle
            x_5,r_5,s_5 \mid x_5^2=r_5^p=s_5^q=1,
            r_5^{s_5}=r_5^k, x_5^{s_5}=x_5^{r_5}=x_5
            \rangle, \\ 
N_6&=\C_p \rtimes \C_{2q} = \langle r_6,s_6
            \mid r_6^p=s_6^{2q}=1, r_6^{s_6}=r_6^{-k} \rangle.
        \end{split}
    \end{align*}
We now describe their automorphism groups.
    \begin{enumerate}[label=\arabic*)]
        \item We generate $\Aut(N_1)$ by $\sigma_1$
          and $\rho_1$, where $\sigma_1(x_1)=x_1^{a_p},
          \rho_1(x_1)=x_1^{a_q}$, and $a_p$ and $a_q$ have orders
          $p-1$ and $q-1$ modulo $2pq$ respectively. So
          $|\Aut(N_1)|=(p-1)(q-1)$.
        \item Note that $\Aut(N_2) \cong \Aut(\C_p) \times
          \Aut(\D_q)$. We generate $\Aut(N_2)$ by 
          $\sigma_2$, $\phi_2$, and $\psi_{b,2}$, where $\sigma_2(x_2)=x_2^a$,
          $\phi_2(s_2)=r_2s_2$ and $\psi_{b,2}(r_2)=r_2^b$ such that
          $a$ and $b$ have orders $(p-1) \bmod{p}$ and $(q-1) \bmod{q}$
          respectively, and the maps fix the remaining
          generators. So $|\Aut(N_2)|=q(p-1)(q-1)$.
        \item Note that $\Aut(N_3) \cong \Aut(\C_q) \times
          \Aut(\D_p)$. We generate $\Aut(N_3)$ by 
          $\sigma_3$, $\phi_3$, and $\psi_{b,3}$, where $\sigma_3(x_3)=x_3^a$,
          $\phi_3(s_3)=r_3s_3$ and $\psi_{b,3}(r_3)=r_3^b$ such that
          $a$ and $b$ have orders $(q-1) \bmod{q}$ and $(p-1) \bmod{p}$
          respectively, and the maps fix the remaining
          generators. So $|\Aut(N_3)|=p(p-1)(q-1)$.
        \item We generate $\Aut(N_4)$ by 
          $\phi_4$, $\psi_{b_p,4}$, and $\psi_{b_q,4}$, where
          $\phi_4(s_4)=r_4s_4$ and $\psi_{b_p,4}(r_4)=r_4^{b_p}$ and
          $\psi_{b_q,4}(r_4)=r_4^{b_q}$ such that $b_p$ and $b_q$ have
          orders $p-1$ and $q-1$ modulo ${pq}$ respectively, and the maps
          fix the remaining generators. So
          $|\Aut(N_4)|=pq(p-1)(q-1)$.
        \item Note that $\Aut(N_5) \cong \Aut(\C_2) \times
          \Aut(\C_p \rtimes \C_q) \cong \Aut(\C_p \rtimes \C_q)$. We
          generate $\Aut(N_5)$ by $\phi_5$ and
          $\psi_{b,5}$, where $\phi_5(s_5)=r_5s_5$ and
          $\psi_{b,5}(r_5)=r_5^b$ such that $b$ has order
          $(p-1)\bmod{p}$, and the maps fix the remaining
          generators. So $|\Aut(N_5)|=p(p-1)$.
        \item We generate $\Aut(N_6)$ by $\phi_6$
          and $\psi_{b,6}$, where $\phi_6(s_6)=r_6s_6$ and
          $\psi_{b,6}(r_6)=r_6^b$ such that $b$ has order
          $(p-1)\bmod{p}$, and the maps fix the remaining
          generators. So $|\Aut(N_6)|=p(p-1)$.
    \end{enumerate}
Observe that $\Hol(N_i)$ for $1 \leq i \leq 6$ has a 
unique Sylow $p$-subgroup. 

\subsection{Transitive subgroups of degree $2pq$}

\begin{proposition}\label{four_types}
Let $G$ be a finite group.
    \begin{enumerate}[label=$\bullet$]
        \item If $(G,\C_{2pq})$ is admissible, then so are $(G,\C_p \times \D_q)$, $(G,\C_q \times \D_p)$ and $(G,\D_{pq})$.
        \item If either $(G,\C_p \times \D_q)$ or $(G,\C_q \times \D_p)$ is admissible, then so is $(G,\D_{pq})$.
    \end{enumerate}
\end{proposition}
\begin{proof}
For $i=2,3,4$, we first 
identify 
regular subgroups 
$M_i$ of $\Hol(N_i)$ which are isomorphic to $C_{2pq}$ and whose
normalisers in $\Perm(N_i)$ are contained in $\Hol(N_i)$. In
particular, $|\Hol(N_1)|=2pq(p-1)(q-1)$, so if 
the normaliser of $M_i$ in $\Hol(N_i)$ has this order, then we
have finished.
Consider the following regular subgroups $M_i$ of $\Hol(N_i)$ such
that $M_i \cong \C_{2pq}$:
    \begin{align*}
        M_2 & = \langle x_2,r_2, (s_2,\psi_{-1,2}) \rangle,\\ 
        M_3 & = \langle x_3,r_3, (s_3,\psi_{-1,3}) \rangle,\\ 
        M_4 & = \langle r_4, (s_4,\psi_{-1,4}) \rangle.
    \end{align*}
    Here $\psi_{-1,i}$ is a power of $\psi_{b,i}$ such that
    $\psi_{-1,i}(r_i)=r_i^{-1}$ for $i \in \{2,3\}$, 
and $\psi_{-1,4}=\psi^a_{b_p,4}\psi^b_{b_q,4}$ where $a=(p-1)/2$ 
and $b=(q-1)/2$, thus $\psi_{-1,4}(r_4)=r_4^{-1}$. 
Clearly, 
\begin{itemize}
\item $\langle x_2,r_2,s_2,\sigma_2,\psi_{b,2}\rangle \leq 
\Hol(N_2)= \Hol(\C_p \times \D_q)$ normalises $M_2$; 

\item 
$\langle x_3,r_3,s_3, \sigma_3,\psi_{b,3}\rangle \leq \Hol(N_3)= 
\Hol(\C_q \times \D_p)$ 
normalises $M_3$; 

\item 
$\langle s_4,r_4, \psi_{b_p,4}, \psi_{b_q,4} \rangle \leq \Hol(N_4) = 
\Hol(\D_{pq})$ normalises $M_4$. 
\end{itemize}
Each normaliser has order $2pq(p-1)(q-1)$.

    We now identify regular subgroups of $\Hol(N_4)$
    isomorphic to $\C_p \times \D_q$ and $\C_q \times \D_p$
    respectively, such that their normalisers in $\Perm(N_4)$ are
    contained in $\Hol(N_4)$. 
Consider 
    \begin{align*}
        \C_p \times \D_q & \cong J_2=\langle r_4, (s_4,\psi^{(p-1)/2}_{b_p,4})
        \rangle,\\ 
        \C_q \times \D_p & \cong J_3=\langle r_4, (s_4,\psi^{(q-1)/2}_{b_q,4}) \rangle.
    \end{align*}
Observe that 
    $\langle r_4,s_4,\psi_{b_p,4},\psi_{b_q,4},\phi_4^p \rangle 
\leq \Hol(N_4)$ has 
    order $2pq^2(p-1)(q-1)$ and normalises $J_2$; 
also $\langle r_4,s_4,\psi_{b_p,4},\psi_{b_q,4},\phi_4^q \rangle \leq 
\Hol(N_4)$ has 
    order $2p^2q(p-1)(q-1)$ and normalises $J_3$. 
Since  these orders
coincide with those of $\Hol(N_2)$ and 
$\Hol(N_3)$ respectively, we have finished.
\end{proof}

\begin{proposition}\label{two_types}
    Let $G$ be a finite group. If both $(G,\C_p \times \D_q)$ and $(G,\C_q \times \D_p)$ are admissible, then $(G,\C_{2pq})$ is also admissible.
\end{proposition}
\begin{proof}
    Let $G$ be a group which embeds transitively in both $\Hol(N_2)=\Hol(\C_p \times \D_q)$
    and $\Hol(N_3)=\Hol(\C_q \times \D_p)$. We show that $G$ embeds transitively in
    $\Hol(N_1)=\Hol(C_{2pq})$. 
Since $|\Hol(N_2)|=2pq^2(p-1)(q-1)$ and
    $|\Hol(N_3)|=2p^2q(p-1)(q-1)$, we deduce that $|G|$
    divides $2pq(p-1)(q-1)=|\Hol(\C_{2pq})|$. 

    We claim that each 
    transitive subgroup of $\Hol(N_3)$ of order $2pq(p-1)(q-1)$
    is isomorphic to $\Hol(N_1)$. Assume this claim. 
If $H \leq \Hol(N_3)$ is 
    transitive of order dividing $2pq(p-1)(q-1)$, 
then we can extend $H$ to a subgroup of
    order $2pq(p-1)(q-1)$ by adding the appropriate generators of
    $\Aut(N_3)$ to $H$. Hence, by the claim, 
    each transitive subgroup of $\Hol(N_3)$ of order dividing
    $2pq(p-1)(q-1)$ must be a subgroup of a group isomorphic to
    $\Hol(N_1)$. 

    Now we prove the claim. Recall that
    \[\Hol(N_3)=\langle x_3,r_3,s_3, 
\sigma_3, \phi_3, \psi_{b,3} \rangle,\]
    where $\Hol(\C_q) = \langle x_3, \sigma_3 \rangle$ 
    has order $q(q-1)$
    and $\Hol(\D_p) = \langle r_3,s_3,\phi_3, \psi_{b,3} \rangle$ has
    order $2p^2(p-1)$. Let $M\leq \Hol(N_3)$ be 
transitive of order 
    $2pq(p-1)(q-1)$. Now $P=\langle r_3, \phi_3 \rangle$ is 
the unique
    Sylow $p$-subgroup of $\Hol(N_3)$, so $M$ must contain a subgroup of
    $P$ of order $p$; it is either $\langle \phi_3 \rangle$
    or $\langle (r_3,\phi_3^u) \rangle$ for some $0 \leq u \leq p-1$. 
    The transitivity of $M$ allows us to ignore $\langle
    \phi_3 \rangle$, so $M$ contains $\langle
    (r_3,\phi_3^u)\rangle$ for some $u$. Further, $M$ must have a Hall
    $p'$-subgroup conjugate to $\langle x_3, s_3,\sigma_3,\psi_{b,3}
    \rangle$. Therefore, up to conjugacy, 
    \[M= \langle (r_3,\phi_3^u), x_3, s_3,\sigma_3,\psi_{b,3} \rangle\]
    for some $0 \leq u \leq p-1$. 
    If $u \not\in \{0, p - 2\}$, then 
    conjugating $(r_3,\phi_3^u)$ by $s_3$ implies that 
    $\phi_3 \in M$. Hence $p^2$ divides $|M|$, a contradiction.

    Two possibilities (up to conjugation) remain: 
    \begin{enumerate}[label=(\roman*)]
        \item $M = \langle x_3,r_3,s_3,\sigma_3,\psi_{b,3} \rangle =
          \langle x_3,r_3,(s_3,\psi_{-1,2}),\sigma_3,\psi_{b,3}
          \rangle \cong \Hol(\C_{2pq});$ 
        \item $M = \langle
          x_3,(r_3,\phi_3^{p-2}),s_3,\sigma_3,\psi_{b,3} \rangle \cong
          \Hol(\C_{2pq})$. \qedhere
    \end{enumerate}
\end{proof}
An immediate consequence is that $(G,N_1)$ is admissible if and
only if both $(G,N_2)$ and $(G,N_3)$ are admissible.

Now suppose that $p \equiv 1 \bmod q$. 
Recall that there are two additional groups: $N_5=\C_2 \times (\C_p
\rtimes \C_q)$ and $N_6=\C_p \rtimes \C_{2q}$.
\begin{proposition}
    Let $G$ be a finite group and let $i \in \{5,6\}$. 
If $(G,N_i)$ is admissible, then for $1 \leq j \leq 6$ with $j \neq i$,
    \[(G,N_j) \text{ is admissible} \Longleftrightarrow |G| \text{ divides } 2pq(p-1).\]
\end{proposition}
\begin{proof}
    Let $G$ be a group that embeds transitively in $\Hol(N_i)$ for
    $i=5$ or $i=6$.

    Suppose that $G$ also embeds transitively in $\Hol(N_j)$ for
    some $j \in \{1,\ldots,6\}\setminus\{i\}$. By Proposition
    \ref{two_types}, we need only assume that $j \in \{4,5,6\}$. 
    Now $|G|$ divides
    $2p^2q(p-1)=\gcd(|\Hol(N_i)|,|\Hol(N_j)|)$. We 
    show that $G$ cannot contain a subgroup of order $p^2$, 
hence $|G|$ must divide $2pq(p-1)$. 
For a contradiction, we suppose otherwise. 
If $M_j$ is the image of the transitive
    embedding of $G$ into $\Hol(N_j)$, 
then $P_j=\langle r_j^q, \phi_j^q \rangle \leq M_j$, 
since it is the unique Sylow
    $p$-subgroup (of order $p^2$) in each case.

    We now make two claims: 

    \medskip
   \noindent
    \textit{Claim $1$: Let $j,j' \in \{4,5,6\}$. 
If $A_j \leq \Aut(\langle r_j^q \rangle)$ and $A_{j'} \leq \Aut(\langle
      r_{j'}^q \rangle)$ have the same order, 
then $M_j \rtimes A_j \cong M_{j'} \rtimes A_{j'}$.}

    \medskip
   \noindent
    \textit{Claim $2$: No transitive subgroup of $\Hol(N_4)$
      is isomorphic to either $\Hol(N_5)$ or $\Hol(N_6)$.}

    \medskip
Assume these claims. Let $A_i=\Aut(\langle r_i \rangle)$ for $i = 5, 6$, 
and note that $\Hol(N_i)=M_i \rtimes A_i$. Since $M_5 \cong M_6$, 
Claim 1 implies that 
\[\Hol(N_5)=M_5 \rtimes A_5 \cong M_6 \rtimes A_6 =\Hol(N_6),\]
a contradiction. Further, if $A_4=\Aut(\langle r_4^q \rangle)$, then 
\[M_4 \rtimes A_4 \cong M_5 \rtimes A_5 = \Hol(N_5).\]
Now $M_4 \rtimes A_4$ is a subgroup of $\Hol(N_4)$ isomorphic to 
$\Hol(N_5)$, which contradicts Claim 2.
Hence, if $G$ contains a subgroup of order $p^2$, then 
it cannot simultaneously embed transitively in $\Hol(N_i)$ and $\Hol(N_j)$.

\ifversion 
    We first prove Claim 1.  Let $M=M_j, M'=M_{j'}$, $A=A_j, A'=A_{j'}$ and $P=P_j, P'=P_{j'}$ with $j,j' \in \{4,5,6\}$ as above.
    Let $z,z'$ be such that $A=\langle z \rangle$ and $A'=\langle z' \rangle$, 
and let $\alpha:M \to M'$ and $\beta:A \to A'$ be isomorphisms with 
$\beta(z)=z'$. Define the map $f:M \rtimes A \to M'\rtimes A'$ by
    \[f((m,a))=(\alpha(m),\beta(a)).\]
    We show that $f$ is an isomorphism. 
    Clearly $f$ is bijective, so we need only check that
    \[f((m_1,a_1)(m_2,a_2))=f((m_1,a_1))f((m_2,a_2)),\]
or equivalently
    \[(\alpha(m_1a_1(m_2)),\beta(a_1a_2))=(\alpha(m_1)\beta(a_1)(\alpha(m_2)),\beta(a_1)\beta(a_2))\]
    for all $m_1,m_2 \in M$ and all $a_1,a_2 \in A$. As $\alpha$ and $\beta$ are isomorphisms, this is the same as checking that 
    \[\alpha(a_1(m_2))=\beta(a_1)(\alpha(m_2)).\]
    Without loss of generality, we need only check this condition 
on generators. Assume that $a_1=z$ 
(hence $\beta(a_1)=\beta(z)=z'$). 
It therefore suffices to check that
    \begin{equation}\label{extension_isom}
        \alpha(z(y))=z'(\alpha(y))
    \end{equation}
    for every generator $y$ of $M$. We may assume that either $y \in 
    \{ r^q,\phi^q \}$, or the order of $y$ is not divisible by $p$. 
If $y \in 
    \{ r^q,\phi^q \}$, then $z(y)=y^d$ for some $1 \leq d \leq p-1$. 
Since $\alpha(P)=P'$ and each of $P$ and $P'$ is abelian, 
(\ref{extension_isom}) holds.
If the order of $y$ is not divisible by $p$, 
then $z$ and $z'$ act trivially on $y$ and $\alpha(y)$ respectively, 
so (\ref{extension_isom}) again holds.

\fi 

    We now prove Claim 2. Note that $\langle x_5,s_5,\psi_{b,5} \rangle \leq \Hol(N_5)$ and $\langle s_6, \psi_{b,6} \rangle \leq \Hol(N_6)$. Each 
    subgroup is isomorphic to $C_{2q} \times C_q$ and
    is transitive on $\langle z_5,s_5 \rangle$ and $\langle s_6 \rangle$ 
respectively (note that each is isomorphic to $N_j/\langle r_j \rangle$). 
We claim that $\Hol(N_4)$
    has no such subgroup. For a contradiction, assume that $M_4 \leq
    \Hol(N_4)$ is transitive and contains a subgroup $J_4
    \cong C_{2q} \times C_q$ which is transitive on $\langle
    r_4^p,s_4\rangle \cong N_4/\langle r_4^q \rangle$. Without loss of
    generality, we may assume that $r_4^q,\phi_4^q \notin J_4$, as
    they are generators of $M_4$ of order $p$. To simplify notation, 
    we relabel $r_4^p$ as $r_4$ and $\phi_4^p$ as $\phi_4$.
    An element of $J_4$ of order $q$ has the form
    $(r_4^a,\phi_4^b\psi_{b_p,4}^c)$. We claim that two such elements,
    say $(r_4^a,\phi_4^b\psi_{b_p,4}^c)$ and 
$(r_4^{a'},\phi_4^{b'}\psi_{b_p,4}^{c'})$,
    commute if and only if one is a power of the
    other, or $c=c'=0$. Consider the equations
    \begin{align*}
        &a(1-b_p^{c'})=a'(1-b_p^c)\\ 
        &b(1-b_p^{c'})=b'(1-b_p^c)
    \end{align*}
    which must be satisfied. If $c=0$ but $c'\neq 0$, then 
    $a=b=c=0$, which contradicts the fact that
    $(r_4^a,\phi_4^b\psi_{b_p,4}^c)$ is an element of order $q$. A
    similar argument holds for $c \neq 0$ and $c'=0$. 
If both $c$ and $c'$ are non-zero, then 
    \[(r_4^a,\phi_4^b\psi_{b_p,4}^c)^x=(r_4^{aX},\phi_4^{bX}\psi_{b_p,4}^{cx}),\]
    where $X=(1-b_p^{cx})/(1-b_p^c)$. If $x=c'/c$, then 
    $(r_4^a,\phi_4^b\psi_{b_p,4}^c)^x=(r_4^{a'},\phi_4^{b'}\psi_{b_p,4}^{c'})$.
 The subgroup of $J_4$ isomorphic to $C_q \times C_q$ must therefore be
    generated by $(r_4^{a_1},\phi_4^{b_1})$ and $(r_4^{a_2},\phi_4^{b_2})$ for
    some $0 \leq a_1,a_2,b_1,b_2 \leq p-1$. 
An element $y_4$ of $M_4$ of order 2 has the form
    $(r_4^as_4^b,\phi_4^c\psi_{b_p,4}^d\psi_{b_q,4}^e)$, where
    $a,b,c,d,e$ are chosen so that $y_4^2=1_{J_4}$. 
In particular, $b=1$, otherwise $J_4$ is not transitive 
on $\langle r_4,s_4 \rangle$. The relations needed for $y_4$ to commute 
with $(r_4^{a_i},\phi_4^{b_i})$ for $i \in \{1,2\}$ imply 
the following equations:
    \begin{align}
        &a_i(1+b_q^e)+b_i=0\label{rel1}\\ 
        &b_i(1-b_q^e)=0.\label{rel2}
    \end{align}
    From (\ref{rel2}), either $b_i=0$ or $e=0$. Suppose first that $b_1=0$. 
Note that $a_1 \not=0$, otherwise $(r_4^{a_1},\phi_4^{b_1})$ is trivial and 
so does not have order
    $q$. In particular, we may assume that $(a_1,b_1)=(1,0)$ and 
$(a_2,b_2)=(0,1)$, and deduce that 
    \begin{align*}
        &a_1(1+b_q^e)=0 \text{ by (\ref{rel1})},\\ 
        &b_2=0 \text{ by (\ref{rel1}), and} \\ 
        &b_2(1-b_q^e)=0 \text{ by (\ref{rel2})}.
    \end{align*}
    But this gives a contradiction: $1=b_2=0$. Hence $b_1 \neq 0$ and 
$b_2 \neq 0$, and so $e=0$. 
This implies that 
    \[2a_i+b_i=0\]
    for $i \in \{1,2\}$, so 
$(r_4^{a_i},\phi_4^{b_i})=(r_4,\phi_4^{-2})^{a_i}$. Hence 
one generator is a power of the other. Thus
$$\langle (r_4^{a_1},\phi_4^{b_1}),(r_4^{a_2},\phi_4^{b_2}) 
\rangle=\langle(r_4,\psi_4^{-2})\rangle \not\cong C_q \times C_q,$$
a contradiction. Hence $\Hol(N_4)$ has no transitive subgroup 
isomorphic to $C_{2q} \times C_q$, so no transitive subgroup of 
$\Hol(N_4)$ is isomorphic to either $\Hol(N_5)$ or $\Hol(N_6)$. 
Therefore $|G|$ divides $2pq(p-1)$.

We now prove the converse. Consider the following subgroups of order $2pq(p-1)$:
    \begin{align}
    \begin{split}\label{2pq(p-1)_grps}
        M_1 & = \langle x_1, \sigma_1 \rangle \leq\Hol(N_1)\\ 
        M_2 & = \langle x_2,r_2,(s_2,\psi_{-1,2}),\sigma_2 \rangle \leq \Hol(N_2)\\ 
        M_3 & = \langle x_3,(r_3,\phi_3^{-2}),s_3,\psi_{b,3} \rangle \leq \Hol(N_3)\\ 
        M_4 & = \langle (r_4,\phi_4^{-2}),s_4,\psi_{b_p,4} \rangle \leq \Hol(N_4)\\ 
        M_5 & = \langle x_5,s_5, (r_5,\psi_{k^{-1},5}), \psi_{b,5} \rangle \leq \Hol(N_5)\\ 
        M_6 & = \langle s_6,(r_6,\psi_{k^{-1},6}),\psi_{b,6} \rangle \leq \Hol(N_6).
    \end{split}
    \end{align}
    Each $M_i$ is isomorphic to $\C_{2q} \times (\C_p \rtimes \C_{p-1})$, and is a
    transitive subgroup of $\Hol(N_i)$ for $1 \leq i \leq 6$. In each
    case, the stabiliser of $1_{N_i}$ is either $\langle \sigma_i
    \rangle$, $\langle \psi_{b,i} \rangle$, or $\langle \psi_{b_p,4} \rangle$; so
    the $M_i$ are pairwise permutation isomorphic.
Hence, if $M \leq M_i$ is a transitive subgroup of
    $\Hol(N_i)$, then $M$ also embeds transitively in $\Hol(N_j)$ for
    $1 \leq j \leq 6$.
    
    For $i \in \{5,6\}$, we claim that each 
   transitive subgroup of $\Hol(N_i)$ having
    order $2pq(p-1)$ is isomorphic to $\C_{2q} \times (\C_p \rtimes
    \C_{p-1})$. Assume this claim. If $H$ is a transitive subgroup of
    $\Hol(N_i)$ of order $2pqd$ (where $d$ is a divisor of $p-1$),
    then we can extend $H$ to a subgroup of order $2pq(p-1)$ by adding the
    appropriate generators of $\Aut(N_i)$ to $H$. 
    Hence, by the claim, each transitive subgroup $H$ of $\Hol(N_i)$ of
    order dividing $2pq(p-1)$ 
    must be a subgroup of a group isomorphic to
    $\C_{2q} \times (\C_p \rtimes \C_{p-1})$. Further, by
    (\ref{2pq(p-1)_grps}), we deduce that $H$ embeds transitively
    in $\Hol(N_j)$ for $1 \leq j \leq 6$.

    We prove the claim for $i=5$; the argument for $i=6$ is similar.
    Let $M$ be a transitive subgroup of $\Hol(N_5)$ of order
    $2pq(p-1)$. Every Hall $p'$-subgroup of $M$ is conjugate
    to $\langle x_5,s_5,\psi_{b,5} \rangle$, and every Sylow
    $p$-subgroup of $M$ is contained in $\langle r_5, \phi_5
    \rangle$. If $M$ is transitive, then it must contain
    $(r_5,\phi_5^u)$ for some $0 \leq u \leq p-1$. Therefore, up to conjugacy,
    $M=\langle (r_5,\phi_5^u), x_5,s_5, \psi_{b,5} \rangle$.  But 
    $s_5(r_5,\phi_5^u)s_5^{-1}=(r_5^{k-u},\phi_5^u)$, and 
hence $r_5^{k-(u+1)}=(r_5^{k-u},\phi_5^u)(r_5,\phi_5^u)^{-1}$. 
If $u \neq k-1$, then $p^2$ divides $|M|$, a contradiction. 
If $u=k-1$, then $M$ is isomorphic to $\C_{2q} \times (\C_p \rtimes \C_{p-1})$.
\end{proof}

\section{Observations and a conjecture}
We observe a few (non-trivial) patterns in
the total number of Hopf--Galois structures on either separable or
Galois extensions (or their relation to skew bracoids and skew braces). 
As one example, consider $n$ where $\gcd(n,\varphi(n))=1$. 
Every group $N$ of order $n$ is cyclic. Byott \cite{Byo96} observed that 
there is a unique Hopf--Galois structure associated to the unique 
transitive subgroup of $\Hol(N)$ of every relevant order. Therefore
the number of (equivalence classes) of skew braces of degree $n$ 
equals the number of Hopf--Galois structures on separable extensions 
of degree $n$.

As we demonstrate in Table \ref{tab:patterns},  
our data for degrees 44, 92, 188 and 236 (all of the form
$4p$ where $p \equiv 11 \bmod 12$) exhibits interesting 
patterns; as does that for the square-free degrees 30, 70, 190 and 230 (all of the form $10p$ where $p \equiv 3 \bmod{4}$ and $p \not\equiv 1 \bmod{5}$).
The number of skew braces agrees with that proved for such 
orders in \cite{AB22} and \cite[\S 7.2]{AB21}.

\begin{table}[htp]
\centering
\small
\begin{tabular}{|l|l|l|l|l|l|l|l|l|}
\hline
Degree & Types & \#HGS & \#Sbracoids & \#Gal & \#Sbraces & \multicolumn{2}{l|}{Almost classical} & \#BC \\
       &       &       &             &       &           & \#HGS & \#Sbracoids & HGS \\
\hline

44 & 4 & 466 & 200 & 184 & 29 & 82 & 70 & 141 \\
92 & 4 & 706 & 200 & 352 & 29 & 82 & 70 & 141 \\
188 & 4 & 1186 & 200 & 688 & 29 & 82 & 70 & 141 \\
236 & 4 & 1426 & 200 & 856 & 29 & 82 & 70 & 141 \\

\hline
\hline
30 & 4 & 479 & 304 & 80 & 36 & 99 & 72 & 197 \\
70 & 4 & 1012 & 608 & 120 & 36 & 198 & 144 & 411 \\
190 & 4 & 1761 & 912 & 240 & 36 & 297 & 216 & 625 \\
230 & 4 & 1444 & 608 & 280 & 36 & 198 & 144 & 411 \\
\hline
\end{tabular}
\caption{Patterns among certain degrees}
\label{tab:patterns}
\end{table}

When we consider the number of structures admitted by specific
types, other patterns emerge. We formulate one 
observation as a conjecture.
\begin{conjecture}
    Let $N$ be a non-abelian finite simple group. The almost
    classically Galois Hopf--Galois structures of type $N$ are exactly
    the Hopf--Galois structures which admit a bijective correspondence.
\end{conjecture}
Martin-Lyons and Truman \cite[Theorem 5.9]{MLT24} demonstrate a 
correspondence between the ideals of a skew bracoid and the intermediate 
fields which are realisable with respect to the associated Hopf--Galois 
structure. This correspondence and the observation that the outer automorphism 
group of a finite simple group is small both support our conjecture.


\section*{Acknowledgements}
Darlington  thanks Roberto Civino, Sam Hodgkinson, 
and Leandro Vendramin for their insight, questions, and familiarity with {\sc Magma} and Python which motivated this paper and greatly assisted with earlier versions of this work.
He was supported by 
the Engineering and Physical Sciences Doctoral Training Partnership
research grant EP/T518049/1 (EPSRC DTP) and 
Project OZR3762 of Vrije Universiteit Brussel and FWO Senior Research
Project G004124N.

O’Brien was supported by the Marsden Fund of New Zealand Grant 23-UOA-080.
\bibliography{MyBib}
\end{document}